\documentclass[12pt]{amsart}
\usepackage{color,graphicx,pstricks}

\newtheorem{theorem}{Theorem}
\theoremstyle{plain}

\newtheorem{corollary}{Corollary}

\newtheorem{definition}{Definition}

\newtheorem{lemma}{Lemma}

\newtheorem{proposition}{Proposition}

\numberwithin{equation}{section}
\newtheorem*{t1}{Result 1}
\newtheorem*{t2}{Result 2}
\newtheorem*{t3}{Result 3}
\newtheorem*{t4}{Result 4}
\newtheorem*{t5}{Result 5}

 \title{Plane overpartitions and cylindric partitions}

 \author{Sylvie Corteel, Cyrille Savelief and Mirjana Vuleti\'{c}}
 \address{CNRS, LIAFA, Case 7014, Universit\'e Paris-Diderot, 75205 Paris Cedex 13, France}\email{corteel@liafa.jussieu.fr}
\address{CNRS, LRI, Universit\'e Paris-Sud, 91405 Orsay Cedex, France}\email{csavelief@gmail.com}
 \address{Department of Mathematics,
Brown University, Providence, RI 02912, USA}\email{vuletic@math.brown.edu}

\thanks{The first author was partially supported by
  the grant ANR-08-JCJC-0011}
 \begin{document}

\maketitle

\begin{abstract}
Generating functions for plane overpartitions are obtained using various methods such as nonintersecting paths, RSK type algorithms and symmetric functions. We extend some of the generating functions to cylindric partitions. Also, we show that plane overpartitions correspond to certain domino tilings and we give some basic properties of this correspondence.
\end{abstract}

\section{Introduction}

The goal of the first part of this paper is to introduce a new
object called plane overpartitions, and to give several enumeration
formulas for these plane overpartitions. A {\it plane overpartition}
is a plane partition where (1) in each row the last occurrence of an
integer can be overlined or not and all the other occurrences of this integer are not overlined and (2) in each column the first
occurrence of an integer can be overlined or not and all the other occurrences of this integer
are overlined. An example of a plane overpartition is
$$
\begin{array}{llll}
4 & 4 & \bar{4} & \bar{3}\\
\bar{4} & 3 & {3} & \bar{3}\\
\bar{4} & \bar{3} & &\\
{3} &  & &\\
\end{array}
$$

This paper takes its place in the series of papers on overpartitions started
by Corteel and Lovejoy \cite{CL}. The motivation is
to show that the generating function for plane overpartitions
is:
\begin{equation}\label{first}
\prod_{n\ge 1}\left(\frac{1+q^n}{1-q^n}\right)^n.
\end{equation}
In this paper, we give several proofs of this result and several
refinements and generalizations. Namely, we prove the following results.
\begin{t1}The hook--content formula for the generating function for plane overpartitions of a given
shape, see Theorem \ref{T1}.
\end{t1}
\begin{t2}The hook formula for the generating function for reverse plane
overpartitions, see Theorem \ref{reverse}.
\end{t2}
\begin{t3}The generating function formula for plane overpartitions with bounded
parts, see Theorem \ref{boundedparts}.
\end{t3}

The goal of the second part of this paper is to extend the
generating function formula for cylindric partitions due to Borodin \cite{Bo} and the following
1-parameter generalized MacMahon's formula due to the third author of this paper \cite{V2}:
\begin{equation}\label{genmac}
\sum_{\substack{\Pi \text{ is a} \\ \text{plane
partition}}}A_\Pi(t)q^{|\Pi|} =\prod_{n=1}^\infty
\left(\frac{1-tq^{n}}{1-q^{n}}\right)^{n},
\end{equation}
where the weight $A_\Pi(t)$ is a polynomial in $t$ that we describe
below.


Given a plane partition $\Pi$ (a Ferrers diagram filled with positive integers that form nonincreasing rows and columns), a connected component of $\Pi$ is the set of all connected boxes of its Ferrers diagram that are filled with a same number. If a box $(i,j)$ belongs to a connected component $C$ then we define its level $h(i,j)$ as the smallest positive integer such that $(i+h,j+h)$ does not belong to $C$. A border component of level $i$ is a connected subset of a connected component whose all boxes have level $i$, see Figure \ref{Fig11}.  We associate to each border component of level $i$, the
weight $(1-t^i)$. The polynomial $A_\Pi(t)$ is the product of the weights of its border components. For the plane partition from Figure \ref{Fig11} it is $(1-t^{10})(1-t^2)^2(1-t^3)^2$.

\begin{figure} [htp!]
\centering
\includegraphics[height=4cm]{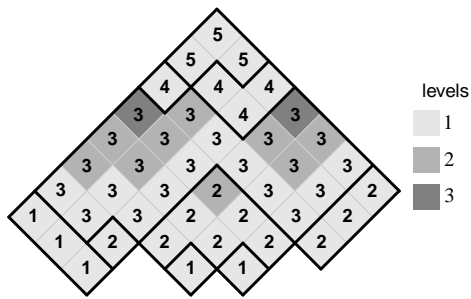}
\caption{A plane partition: border components} \label{Fig11}
\end{figure}

We give a new proof of the 1-parameter generalized MacMahon's formula. 
We also extend this formula  to two more general objects:
skew plane partitions and cylindric partitions. Namely, we prove the following results.
\begin{t4}1-parameter generalized formula for the generating function for skew plane partitions, see Theorem \ref{skew}.
\end{t4}
\begin{t5}1-parameter generalized formula for the generating function for cylindric partitions, see Theorem
\ref{uvodcyl}.
\end{t5}

In the rest of this section we give definitions and explain our
results in more detail.


A partition $\lambda$ is a nonincreasing sequence of positive
integers $(\lambda_1,\ldots ,\lambda_k)$. Each $\lambda_i$ is
a part of the partition and the number of parts is denoted by
$\ell(\lambda)$. The weight $|\lambda|$ of $\lambda$ is the sum of
its parts.  A partition $\lambda$ can be graphically represented by
the Ferrers diagram that is a diagram formed of $\ell(\lambda)$ left
justified rows, where the $i^{th}$ row consists of $\lambda_i$ cells
(or boxes). The conjugate of a partition $\lambda$, denoted by
$\lambda'$, is a partition that has the Ferrers diagram equal to the
transpose of the Ferrers diagram of $\lambda$. For a cell $(i,j)$ of
the Ferrers diagram of $\lambda$ the {\em hook length} of this cell is
$h_{i,j}=\lambda_i+\lambda'_j-i-j+1$ and the {\em content} is $c_{i,j}=j-i$.
It is well known that the generating function for partitions that
have at most $n$ parts is $1/(q)_n$, where
$(a)_n:=(a;q)_n=\prod_{i=0}^{n-1}(1-aq^i)$. More definitions on
partitions can be found, for example, in \cite{A} or \cite{Mac}.

An {\it overpartition} is a partition where the last occurrence
of an integer can be overlined \cite{CL}. Last occurrences in an overpartition are in one--to--one correspondence with corners of the Ferrers diagram and overlined parts can be represented by marking the corresponding corners. The generating function for overpartitions that have at most $n$ parts is
$(-q)_n/{(q)_n}$.

Let $\lambda$ be a partition. A {\it plane partition} of shape
$\lambda$ is a filling of cells of the Ferrers diagram of $\lambda$
with positive integers that form a nonincreasing sequence along each
row and each column. We denote the shape of a plane partition $\Pi$
by $\text{sh}(\Pi)$ and the sum of all entries by $|\Pi|$,
called the weight of $\Pi$. It is well known, under the name of
MacMahon's formula, that the generating function for plane partitions
is
\begin{equation}
\label{planeovergf}
\sum_ {\substack{\Pi \text{ is a}\\ \text{plane
partition}}}q^{|\Pi|}=
\prod_{i=1}^\infty\left(\frac{1}{1-q^i}\right)^i.
\end{equation}

One way to prove this is to construct a bijection between plane partitions and
pairs of semi-standard Young tableaux of the same shape
and to use the RSK correspondence between these Young tableaux and certain 
matrices \cite{BK}.

%

Recall that a plane overpartition is a plane partition where in each row the last occurrence of an integer can be overlined or not and in each column the first occurrence
of an integer can be overlined or not and all the others are overlined.
This definition  implies that the entries
strictly decrease along diagonals, i.e. all connected components are also border components.
Therefore, a plane overpartition is
a diagonally strict plane partition where  some entries are
overlined. More precisely, it is easy to check that inside a border component 
only one entry can be chosen to be overlined or not and this entry
is the upper right entry.

Plane  overpartitions are therefore
in bijection with diagonally strict plane partitions
where each border component can be overlined or not (or weighted by 2).
Recently, those weighted diagonally strict plane partitions were studied in
\cite{FW,FW1,V1,V2}. The first result obtained was the shifted MacMahon's formula that says that the generating function for plane overpartitions is indeed equation \eqref{first}.
This was obtained as a limiting case of the generating function formula for plane overpartitions which fit into an $r\times c$ box, i.e. whose shapes are contained in the rectangular  shape with $r$ rows and $c$ columns.
\begin{theorem}\cite{FW,V1}\label{osnovna}
The generating function for plane overpartitions which fit in an
 $r\times c$ box is
$$
\prod_{i=1}^r\prod_{j=1}^c \frac{1+q^{i+j-1}}{1-q^{i+j-1}}.
$$
\end{theorem}
This theorem was proved in  \cite{FW,V1} using Schur $P$ and $Q$ symmetric functions and a suitable Fock space. In \cite {V2} the theorem was proved in a bijective way where an RSK--type algorithm (due to Sagan \cite{Sa}, see also
Chapter XIII of \cite{HH})
was used to construct a bijection between plane overpartitions
and matrices of nonnegative
integers where positive entries can be overlined.

In Section 2, we give a mostly combinatorial proof of the
generalized MacMahon formula \cite{V2}. Namely, we prove:
\begin{theorem}\cite{V2}
\begin{equation}
\sum_{\Pi\in {\mathcal P}(r,c)}A_\Pi(t)q^{|\Pi|}
=\prod_{i=1}^r\prod_{j=1}^c \frac{1-tq^{i+j-1}}{1-q^{i+j-1}}.
\end{equation}
\label{mirjana}
\end{theorem}
In the above formula,  ${\mathcal P}(r,c)$ is the set of plane partitions with at
most $r$ rows and $c$ columns.  When we set $t=-1$, only the border
components of level $1$ have a non--zero weight and we get back Theorem
\ref{osnovna}.

The main result of Section 3 is a hook--content formula for the
generating function for plane overpartitions of a given shape. More
generally, we give a weighted generating function where overlined
parts are weighted by some parameter $a$. 

Let $\mathcal{S}(\lambda)$ be the set of all plane overpartitions of
shape $\lambda$.  The number of overlined parts of a plane overpartition
$\Pi$ is denoted by $o(\Pi)$. 
For example, $\Pi=
\begin{array}{llll}
4 & 4 & \bar{4} & \bar{3}\\
\bar{4} & 3 & {3} & \bar{3}\\
\bar{4} & \bar{3} & &\\
\end{array}
$
is a plane overpartition of shape $(4,4,2)$, with $|\Pi|=35$ and $o(\Pi)=6$.

\begin{theorem}\label{T1}
Let $\lambda$ be a partition. The weighted generating function for
plane overpartitions of shape $\lambda$ is
\begin{equation}
\sum_{\Pi \in \mathcal{S}(\lambda)} a^{o(\Pi)}q^{|\Pi|}=q^{\sum_i
i\lambda_i}\prod_{(i,j)\in\lambda}\frac{1+aq^{c_{i,j}}}{1-q^{h_{i,j}}}.
\end{equation}
\end{theorem}

We prove this theorem using 
 using a correspondence between plane overpartitions and sets of nonintersecting paths that use three kinds of steps.
(The work of Brenti used similar paths to compute super Schur functions \cite{Br}.) Another way to prove
this result  is to show that  plane overpartitions of shape $\lambda$ are
in bijection with super semistandard
tableaux of shape $\lambda$. This is presented in a remark in Section \ref{nonint}.


We also give the weighted generating function formula for plane overpartitions
``bounded'' by $\lambda$, where by that we mean
plane overpartitions such that the $i^{th}$ row of the plane
overpartition is an overpartition that has at most $\lambda_{i}$
parts and at least $\lambda_{i+1}$ parts. Let $\mathcal{B}(\lambda)$ be the set of all such plane overpartitions.
\begin{theorem}\label{T2}
Let $\lambda$ be a partition. The weighted generating function for
plane overpartitions such that the $i^{th}$ row of the plane
overpartition is an overpartition that has at most $\lambda_{i}$
parts and at least $\lambda_{i+1}$ parts is
\begin{equation}
\sum_{\Pi \in \mathcal{B}(\lambda)} a^{o(\Pi)}q^{|\Pi|}=q^{\sum_i
(i-1)\lambda_i}\prod_{(i,j)\in\lambda}\frac{1+aq^{c_{i,j}+1}}{1-q^{h_{i,j}}}.
\end{equation}
\end{theorem}

Note that it is enough to assign weights to overlined (or nonoverlined) parts only because generating functions where overlined and nonoverlined parts are weighted by $a$ and $b$, respectively, follow trivially from the above formulas.



The number of nonintersecting paths is given by  a determinantal formula (Lemma 1 of  \cite{Li}). This  result was anticipated by Lindstr\"om \cite{Li} and  Karlin and McGregor \cite{KM1,KM2}, but Gessel and Viennot were first to use it for enumerative purpose of various classes of plane partitions \cite{GV, GV1}. Applying the result and evaluating the determinants we obtain hook--content formulas.  We use a simple involution to show that the Stanley hook--content formula
(Theorem 7.21.2 of \cite{St2}) follows from our formula.

From the symmetric function point of view, these formulas are given by Schur functions in a difference of two alphabets, as explained in Section 3.

The end of Section 3 is devoted to {\it reverse} plane overpartitions. 
A reverse plane partition of shape $\lambda$ is a filling of cells of the Ferrers diagram of $\lambda$ with nonnegative integers that form a nondecreasing sequence along each row and each column. A reverse plane overpartition is a reverse plane partition where (1) only positive entries can be overlined, (2) in each row the last occurrence of an integer can be overlined or not and (3) in each column the first occurrence of a positive integer can be overlined or not and all others (if positive) are overlined.
An example of a reverse plane overpartition is
$$
\begin{array}{llllll}
0 & 0 & 3 & 4 & 4 & \bar{4}\\
0 & 0 & 4 & \bar{4} &&\\
1&\bar{3} & & &&\\
{3} & \bar{3} & &&\\
\end{array}
$$

It was proved by Gansner \cite{G} that the generating function for reverse plane partitions of a given shape $\lambda$ is
\begin{equation}
\prod_{(i,j)\in \lambda}\frac{1}{1-q^{h_{i,j}}}.
\label{gans}
\end{equation}

Let $\mathcal{S}^R(\lambda)$ be the set of all reverse plane overpartitions of shape $\lambda$. The generating function for reverse plane overpartitions is given by the following hook formula.
\begin{theorem} Let $\lambda$ be a partition. The generating function for reverse plane overpartitions of shape $\lambda$ is
$$
\sum_{\Pi \in \mathcal{S}^R(\lambda)}q^{|\Pi|}=\prod_{(i,j)\in
\lambda}\frac{1+q^{h_{i,j}}}{1-q^{h_{i,j}}}.
$$
\label{reverse}
\end{theorem}
We construct a bijection between reverse plane overpartitions of a given shape and sets of nonintersecting paths whose endpoints are not fixed. Using results of \cite{St} we obtain a Pfaffian formula for the generating function for reverse plane partitions of a given shape. Subsequently, we evaluate the Pfaffian and obtain a  proof of the hook formula. When $\lambda$ is the partition with $r$ parts equal to $c$, this result is the generating function formula for plane overpartitions fitting in an $r\times c$, box namely Theorem \ref{osnovna}.

In Section 4 we make a connection between plane overpartitions and domino tilings. We give some basic properties of this correspondence, such as how a removal of a box or an overline changes the corresponding tiling. This correspondence connects a measure on strict plane partitions studied in \cite{V2} to a measure on domino tilings. This connection was expected by similarities in correlation kernels, limit shapes and some other features of these measures, but the connection was not established before.

In Section 5 we propose a bijection between matrices and pairs of plane overpartitions based on ideas of 
Berele and Remmel \cite{berele:remmel:1985}. We give another stronger version of the shifted MacMahon's formula, as we give a weighted generating function for plane overpartitions with bounded entries. Let $\mathcal{L}(n)$ be the set of all plane overpartitions with the largest entry at most $n$.
\begin{theorem}\label{boundedparts}
The weighted generating functions for plane overpartitions where the
largest entry is at most $n$ is
$$
\sum_{\Pi \in
\mathcal{L}(n)}a^{o(\Pi)}q^{|\Pi|}=\prod_{j=1}^n \frac{\prod_{i=0}^n(1+aq^{i+j})}{\prod_{i=1}^{j}(1-q^{i+j-1})(1-a^2q^{i+j})}.
$$
\end{theorem}

In Section 6 we study interlacing sequences and cylindric
partitions. We say that a sequence of partitions
$\Lambda=(\lambda^1,\dots,\lambda^T)$ is \textit{interlacing} if
$\lambda^i/\lambda^{i+1}$ or $\lambda^{i+1}/\lambda^i$ is a
horizontal strip, i.e. a skew shape having at most one cell in each
column. Let $A=(A_1,\dots,A_{T-1})$ be a sequence of 0's and 1's. We
say that an interlacing sequence
$\Lambda=(\lambda^1,\dots,\lambda^T)$ has \textit{profile} $A$ if
when $A_i=1$, respectively $A_i=0$, then $\lambda^i/\lambda^{i+1}$,
respectively $\lambda^{i+1}/\lambda^{i}$ is a horizontal strip.
Interlacing sequences are generalizations of plane partitions. Indeed plane partitions  
are interlacing sequences with
profile $A=(0,\ldots,0,1,\ldots ,1)$.
\begin{figure} [htp!]
\centering
\includegraphics[height=6cm]{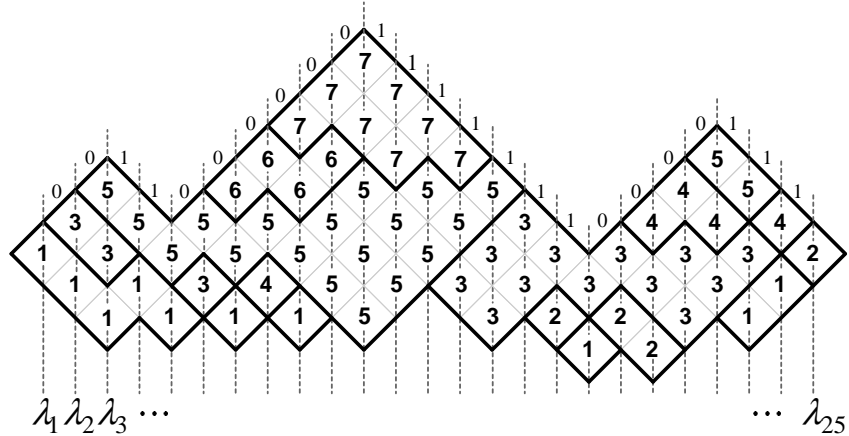}
\caption{The diagram of an interlacing sequence} \label{Fig1}
\end{figure}


We now define the \textit{diagram} of an interlacing sequence. See Figure \ref{Fig1}. We start with a square grid; we denote the two directions defined by the grid lines with 0 and 1. A profile  $A=(A_1,\dots,A_{T-1})$ is represented by a path of length $T+1$ on this grid where the path consists of grid edges whose directions are given by $0,A_1,A_2,\dots,A_{T-1},1$. This path forms the (upper) border of the diagram. Excluding the endpoints of the path we draw the diagonal rays (which form $45^{\circ}$ angles with grid lines) starting at the vertices of the path and we index them (from left to right) with integers from $1$ to $T$. A diagram is a connected subset of boxes of a square grid whose (upper) border is given by the profile path and along the $i^{th}$ diagonal ray there are $\ell(\lambda^i)$ boxes. The filling numbers on the $i^{th}$ diagonal ray are parts of $\lambda^i$ with the largest part at the top. Observe that by the definition of interlacing sequences we obtain monotone sequences of numbers in the direction of grid lines.

A (skew) plane partition and cylindric partition are examples of
interlacing sequences. A plane partition can be written as
$\Lambda=(\emptyset, \lambda^1,\dots,\lambda^T,\emptyset)$ with
profile $A=(0,0,\dots,0,1,\dots,1,1)$ and $\lambda^i$s are diagonals
of the plane partition. A skew plane partition is an interlacing
sequence $\Lambda=(\emptyset,\lambda^{1},\dots,\lambda^T,\emptyset)$
with a profile $A=(0,A_1,\dots,A_{T-1},1)$.  A cylindric partition is an interlacing sequence
$\Lambda=(\lambda^0,\lambda^1,\dots,\lambda^{T})$ where
$\lambda^0=\lambda^T$, and $T$ is called the period of $\Lambda$. A
cylindric partition can be represented by the \textit{cylindric
diagram} that is obtained from the ordinary diagram by
identification of the first and last diagonal.

A {\it connected component} of an interlacing sequence $\Lambda$ is
the set of rookwise connected boxes of its diagram that are
filled with a same number. We denote the number of connected
components of $\Lambda$ with $k(\Lambda)$. For the example from
Figure \ref{Fig1} we have $k(\Lambda)=18$ and its connected
components are shown in Figure \ref{Fig1} (bold lines represent
boundaries of these components).

If a box
$(i,j)$ belongs to a connected component $C$ then we define its {\it
level} $\ell(i,j)$  as the smallest positive integer such that
$(i+\ell,j+\ell)$ does not belong to $C$.
In other words, a level represents the distance from the ``rim'', distance being measured diagonally.
A {\it border component} is
a rookwise connected subset of a connected component where all boxes have the
same level. We also say that this border component is of this level.  For the example from Figure \ref{Fig1}, border components and their levels are shown
in Figure \ref{Fig2} (different levels are represented by different colors).

\begin{figure} [htp!]
\centering
\includegraphics[height=4cm]{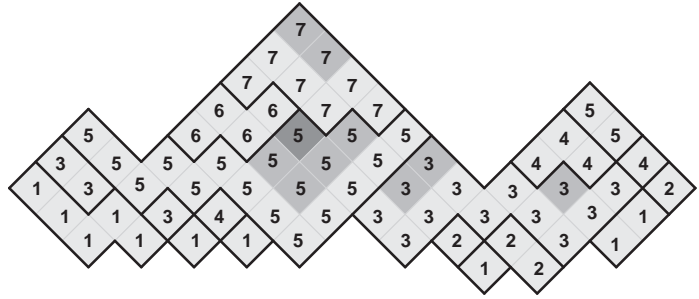}
\caption{Border components with levels} \label{Fig2}
\end{figure}

Let
$(n_1,n_2,\dots)$  be a sequence of nonnegative integers where $n_i$ is the number of $i$--level border
components of $\Lambda$. We
set
\begin{equation}
A_\Lambda(t)=\prod_{i\geq1}(1-t^{i})^{n_i}.
\label{Phi}
\end{equation}
For the example above $A_\Lambda(t)=(1-t)^{18}(1-t^2)^4(1-t^3)$.

For a cylindric partition $\Pi$, we define {\it cylindric connected
components} and {\it cylindric border components} in the same way
but connectedness is understood on the cylinder, i.e. boxes are
connected if they are rookwise connected in the cylindric diagram.
We define
$$
A_\Pi^{\text{cyl}}(t)=\prod_{i\geq1}(1-t^{i})^{n^{\text{cyl}}_i},
$$
where $n^{\text{cyl}}_i$ is the number of cylindric border components of level $i$.

In Section 6 we give a generating function formula for skew plane partitions.
Let $\text{Skew}(T,A)$ be the set of all skew plane partitions
$\Lambda=(\emptyset,\lambda^1,\dots,\lambda^T,\emptyset)$ with
profile $A=(A_0,A_1,\dots,A_{T-1},A_T)$, where $A_0=0$ and $A_T=1$.
\begin{theorem}(Generalized MacMahon's formula for skew
plane partitions; Hall-Littlewood case)
\begin{equation*}
\sum_{\Pi \in \text{Skew}(T,A)}A_\Pi(t)q^{|\Pi|}=
\prod_{\substack{0\leq i< j\leq
T\\A_i=0,\,A_j=1}}\frac{1-tq^{j-i}}{1-q^{j-i}}.
\end{equation*}
\label{skew}
\end{theorem}

Note that as profiles are words in $\{0,1\}$, a profile
$A=(A_0,\dots,A_{T})$ encodes the border of a Ferrers diagram
$\lambda$. Skew plane partitions of profile $A$ are in one--to--one
correspondence with reverse plane partitions of shape $\lambda$.
Moreover, one can check that
$$
\prod_{\substack{0\leq i< j\leq
T\\A_i=0,\,A_j=1}}\frac{1-tq^{j-i}}{1-q^{j-i}}= \prod_{(i,j)\in
\lambda}\frac{1-tq^{h_{i,j}}}{1-q^{h_{i,j}}}.
$$
Therefore the theorem
of Gansner (equation \eqref{gans}) is Theorem \ref{skew} with $t=0$ and
our Theorem \ref{reverse} on reverse plane
overpartitions  is Theorem \ref{skew} with $t=-1$.

This theorem is also a generalization of results of Vuleti\'{c}
\cite{V2}. In \cite{V2} a 2--parameter generalization of MacMahon's formula related to Macdonald symmetric functions was given and the formula is especially simple in the Hall-Littlewood case. In the Hall-Littlewood case, this is a generating function formula for plane partitions weighted by $A_{\Pi}(t)$. Theorem \ref{skew} can be naturally generalized to the Macdonald case, but we do not pursue this here. 

Let $\text{Cyl}(T,A)$ be the set of all cylindric partitions with period $T$ and profile $A=(A_1,\dots,A_T)$. The main result of Section 6 is:
\begin{theorem}(Generalized MacMahon's formula for cylindric partitions; Hall-Littlewood case)\label{uvodcyl}
\begin{equation*}
\sum_{\Pi \in
\text{Cyl}(T,A)}A^{\text{cyl}}_\Pi(t)q^{|\Pi|}=\prod_{n=1}^{\infty}\frac{1}{1-q^{nT}}
\prod_{\substack{1\leq i,j\leq
T\\A_i=0,\,A_j=1}}\frac{1-tq^{(j-i)_{(T)}+(n-1)T}}{1-q^{(j-i)_{(T)}+(n-1)T}},
\end{equation*}
where $i_{(T)}$ is the smallest positive integer such that $i\equiv
i_{(T)} \text{ mod }T$.
\end{theorem}

The case $t=0$ is due to Borodin and represents a generating function formula for cylindric partitions. Cylindric partitions were introduced and enumerated
by Gessel and Krattenthaler \cite{GK}. The result of Borodin could be
also proven using Theorem 5 of \cite{GK} and the ${\rm SU}(r)$-extension
of Bailey's $_6\psi_6$ summation due to Gustafson (equation (7.9)
in \cite{GK}) \cite{Kr}. Again Theorem \ref{uvodcyl} can be naturally
generalized to the Macdonald case. The trace generating function
of those cylindric partitions could also be easily derived from our proof,
as done by Okada \cite{O} for the reverse plane partitions case.
\\

The paper is organized as follows. In Section 2 we
give a mostly combinatorial proof of the generalized MacMahon formula. In Section 3 we use nonintersecting paths and obtain the hook--length formulas for plane overpartitions and reverse plane partitions of a given shape. In Section 4 we make the connection between tilings and plane overpartitions.
In Section 5 we construct a bijection between matrices and pairs of plane overpartitions and obtain a generating function formula for plane overpartitions with bounded part size. In Section 6 we give the hook formula for reverse plane partitions contained in a given shape and the 1--parameter generalization of the generating function formula for cylindric partitions.  
Section 7 contains some concluding remarks.\\

{\bf Acknowledgment.} The authors want to thank Alexei Borodin, the
advisor of the third author, for help and guidance and C\'{e}dric Boutillier, Dominique
Gouyou-Beauchamps, Richard Kenyon and Jeremy Lovejoy for
useful discussions. The authors also want to thank one of the anonymous
referee for his long list of constructive comments.

\section{Plane partitions and Hall--Littlewood  functions}


In this section, we give an alternative proof of the generalization
of MacMahon's formula due to the third author \cite{V2}. Our proof is
mostly combinatorial as it uses a bijection between plane partitions
and pairs of strict plane partitions of the same shape and the
combinatorial description of Hall--Littlewood functions
(\cite{Mac}, Chapter III, Equation (5.11)).

Let ${\mathcal P}(r,c)$ be the set of plane partitions with at most
$r$ rows and $c$ columns. Given a plane partition $\Pi$, let
$A_\Pi(t)$ be the polynomial defined in \eqref{Phi}, as
$A_\Pi(t)=\prod_{r\ {\rm border}\ {\rm component}}(1-t^{{\rm level}(r)})$.

Recall that Theorem \ref{mirjana} states that
$$
\sum_{\Pi\in {\mathcal P}(r,c)}A_\Pi(t)q^{|\Pi|}
=\prod_{i=1}^r\prod_{j=1}^c \frac{1-tq^{i+j-1}}{1-q^{i+j-1}}.
$$

Any plane partition $\Pi$ is in bijection with a sequence of
partitions $(\pi^{(1)},\pi^{(2)},\ldots)$. This sequence is such
that $\pi^{(i)}$ is the shape of the entries greater than or equal
to $i$ in $\Pi$ for all $i$. For example if
$$\Pi=\begin{array}{l}
4433\\
3332\\
1 \end{array},$$
the corresponding sequence is
$((4,4,1),(4,4), (4,3), (2))$.

Note that the plane partition $\Pi$ is column strict if and only
if  $\pi^{(i)}/
\pi^{(i+1)}$ is a horizontal strip for all $i$.

We use a bijection between pairs of column strict
plane partitions $(\Sigma,\Lambda)$ and plane partitions $\Pi$ due to Bender and Knuth \cite{BK}.
We suppose that $(\Sigma,\Lambda)$ are
of the same shape $\lambda$
and that the corresponding sequences are $(\sigma^{(1)},\sigma^{(2)},\ldots)$
and $(\lambda^{(1)},\lambda^{(2)},\ldots)$

Given a plane partition $\Pi=(\Pi_{i,j})$, we define the entries of
diagonal $x$ to be the partition $(\Pi_{i,j})$ with $i,j\ge 1$ and
$j-i=x$. The bijection is such that the entries of diagonal $x$ of
$\Pi$ are $\sigma^{(x+1)}$ if $x\ge 0$ and $\lambda^{(-x-1)}$
otherwise. Note that as $\Lambda$ and $\Sigma$ have the same shape,
the entries
of the main diagonal ($x=0$) are $\sigma^{(1)}=\lambda^{(1)}$.

For example, start with
$$
\Sigma=\begin{array}{l}
4444\\
2221\\
111\end{array} \text{ and } \Lambda=\begin{array}{l}
4433\\
3322\\
111\end{array},
$$
whose sequences are $((4,4,3),(4,3), (4), (4))$
and $((4,4,3),(4,4), (4,2), (2))$, respectively  and get
$$
\Pi=\begin{array}{l}
4444\\
443\\
443\\
22\end{array}.
$$

This construction implies that:
\begin{eqnarray*}
|\Pi|&=&|\Sigma|+|\Lambda|-|\lambda|\\
A_\Pi(t)&=&\frac{\varphi_{\Sigma}(t)\varphi_{\Lambda}(t)}{b_{\lambda}(t)}.
\end{eqnarray*}
Here we have
$$b_\lambda(t)=\prod_{i\ge 1}\varphi_{m_i(\lambda)}(t), \ \ \  \varphi_r(t)=\prod_{j=1}^{r}(1-t^j),$$
and
$$
\varphi_{\Lambda}(t)=\prod_{i\ge 1} \varphi_{\lambda^{(i)}/
\lambda^{(i+1)}}(t).
$$
Moreover
$$
\varphi_\theta(t)=\prod_{i\in I}(1-t^{m_i(\lambda)}),
$$
where $\theta$ is a horizontal strip $\lambda/\mu$,$m_i(\lambda)$ is the multiplicity of $i$ in $\lambda$ and $I$
is the set of integers such that $\theta'_i=1$ and
$\theta'_{i+1}=0$. See \cite{Mac}, Chapter III Sections 2 and 5.

Indeed, the following statements are true.
\begin{itemize}
\item Each factor $(1-t^i)$ in $b_{\lambda}(t)$ is in one-to-one
correspondence with a border component of level $i$ that goes through the
main diagonal  of $\Pi$.
\item Each factor $(1-t^i)$ in $\varphi_{\Sigma}(t)$ is in one-to-one
correspondence with a border component
of level $i$ that ends in a non-negative
diagonal.
\item Each factor $(1-t^i)$ in $\varphi_{\Lambda}(t)$ is in one-to-one
correspondence with a border component of level $i$ that starts in a non-positive
diagonal.
\end{itemize}

Continuing with our example,
we have
$$
\varphi_\Sigma(t)=(1-t)^2(1-t^2),\ \ \ \varphi_\Lambda(t)=(1-t)^3(1-t^2),\ \ \
b_{\lambda}(t)=(1-t)^2(1-t^2),
$$
and
$$
A_\Pi(t)=(1-t)^3(1-t^2)=\frac{(1-t)^2(1-t^2)(1-t)^3(1-t^2)}{(1-t)^2(1-t^2)}.
$$

We recall the combinatorial definition of the Hall--Littlewood
functions following Macdonald \cite{Mac}. The
Hall-Littlewood function $Q_\lambda(x;t)$ can be defined as
$$
Q_\lambda(x;t)=\sum_{\begin{subarray}{c}\Lambda \\
{\rm sh}(\Lambda)=\lambda\end{subarray}}\varphi_\Lambda(t)x^\Lambda;
$$
where $x^\Lambda=x_1^{\alpha_1}x_2^{\alpha_2}\ldots$ and $\alpha_i$ is the number
of entries equal to $i$ in $\Lambda$. See \cite{Mac} Chapter III,
equation (5.11).

A direct consequence of the preceding bijection is that the entries
of $\Sigma$ are less than or equal to $c$, and the entries of $\Lambda$
are less than or equal to $r$ if and only if  $\Pi$ is in ${\mathcal
P}(r,c)$. Therefore~:
$$
\sum_{\Pi\in {\mathcal P}(r,c)}A_\Pi(t)q^{|\Pi|}=
\sum_{\lambda}\frac{Q_{\lambda}(q,\ldots,q^r,0,\ldots;t)
Q_{\lambda}(q^0,\ldots,q^{c-1},0,\ldots;t)}{b_{\lambda}(t)}.
$$

Finally, we need equation (4.4) in Chapter III of \cite{Mac}.
\begin{equation}
\sum_{\lambda} \frac{Q_\lambda(x;t)Q_\lambda(y;t)}{b_{\lambda}(t)}=
\prod_{i,j}\frac{1-tx_iy_j}{1-x_iy_j};
\label{cauchy}
\end{equation}

With the substitutions $x_i=q^i$ for $1\le i\le r$
and $0$ otherwise and   $y_j=q^{j-1}$ for $1\le j\le c$
and $0$ otherwise, we get the result.

\section{Nonintersecting paths}\label{nonint}

\subsection{Plane overpartitions of a given shape}\label{3.1}

In this section we represent plane overpartitions as nonintersecting paths. 
We use the determinantal formula for the number of nonintersecting paths (see \cite{GV1,KM1,KM2,Li} ).  
Evaluating these determinants we obtain the hook--content formulas from Theorems
\ref{T1} and \ref{T2}. A similar approach was used for example in \cite{Br} to
compute super Schur functions.

%

We construct a bijection between the set of paths from $(0,0)$ to
$(x,k)$ and the set of overpartitions with at most $k$ parts and the largest  part at most $x$. 
Given an overpartition the corresponding path
consists of North and East edges that form the border of the Ferrers
diagram of the overpartition except for corners containing an
overlined entry where we substitute a pair of North and East edges
with an North--East edge. For example, the path corresponding
to the overpartition $(6,\bar{6},4,4,\bar{3})$
is shown in
Figure \ref{path}. Note that this construction appears also in
Proposition 2.2 of \cite{Br}.

\begin{figure} [htp!]
\centering
\includegraphics[height=3cm]{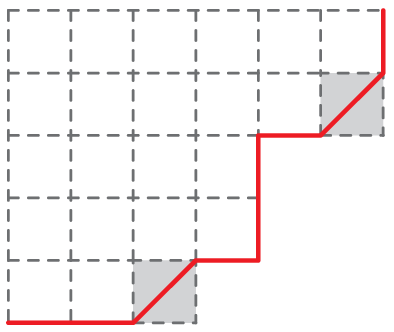}
\caption{Paths and overpartitions} \label{path}
\end{figure}

To each overpartition $\lambda$ we associate a weight equal to
$a^{o(\lambda)}q^{|\lambda|}$, where $o(\lambda)$ is the number of overlined parts. To have the same weight on the
corresponding path  we
introduce the following weights on edges. We assign weight 1 to East
edges, $q^i$ to North edges on (vertical) level $i$ and weight $aq^{i+1}$ to
North-East edges joining vertical levels $i$ and $i+1$.
The weight of the path is equal to the product of weights of its edges.

We will need the following lemma.
\begin{lemma}\cite{CL}  \label{lemaop}The generating function for overpartitions with at
most $k$ parts is given by
\begin{equation}\label{genoverpartitions}
\sum_{l(\lambda)\leq
k}{a^{o(\lambda)}q^{|\lambda|}}=\frac{(-aq)_k}{(q)_k}
\end{equation}
and the generating function for overpartitions with exactly $k$ parts is given by
\begin{equation}\label{genoverpartitionsprim}
\sum_{l(\lambda)=
k}{a^{o(\lambda)}q^{|\lambda|}}=q^k\frac{(-a)_k}{(q)_k}.
\end{equation}
\end{lemma}



For a plane overpartition $\Pi$ of shape $\lambda$ we construct a
set of nonintersecting paths using paths from row overpartitions
where the starting point of the path corresponding to the $i$th row
is shifted upwards by $\lambda_1-\lambda_i+i-1$ so that the starting
point is $(0,\lambda_1-\lambda_i+i-1)$. In that way, we obtain a
bijection between the set of nonintersecting paths from
$(0,\lambda_1-\lambda_i+i-1)$ to $(x,\lambda_1+i-1),$ where $i$ runs
from 1 to $\ell(\lambda)$, and the set of plane overpartitions whose
$i$th row  has at most $\lambda_i$ parts and at least
$\lambda_{i-1}$ parts with $x$ greater or equal to the largest part.
The weights of this set of nonintersecting paths (the product of weights of its paths) is equal to
$a^{o(\Pi)}q^{|\Pi|}$. Figure \ref{Fig4} (see also Figure
\ref{Fig6}) shows the corresponding set of nonintersecting paths for
$x=8$ and the plane overpartition
$$
\begin{array}{lllll}
7&4&\bar{3}&2&\bar{2}\\
3&3&\bar{3}&\bar{2}\\
\bar{3}&2&\bar{1}&&\\
2&&&&
\end{array}.
$$

\begin{figure} [htp!]
\centering
\includegraphics[height=4cm]{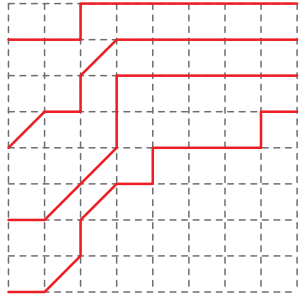}
\caption{Nonintersecting paths} \label{Fig4}
\end{figure}


\begin{definition}\label{DefM}
For a partition $\lambda$ we define $M_\lambda(a;q)$ to be the
$\ell(\lambda)\times \ell(\lambda)$ matrix whose $(i,j)$th entry is
given by
$$
\frac{(-a)_{\lambda_j+i-j}}{(q)_{\lambda_j+i-j}}.
$$
\end{definition}

For a partition $\lambda$ let $\mathcal{B}(\lambda)$ and $\mathcal{S}(\lambda)$ be as in the introduction, i.e. the sets of all plane overpartitions bounded by shape $\lambda$ and of the shape $\lambda$, respectively.

\begin{proposition}\label{bounded} Let $\lambda$ be a partition. The weighted generating function for plane overpartitions whose
$i^{th}$ row is an overpartition that has
at most $\lambda_{i}$ parts and at least $\lambda_{i+1}$ parts  is
given by
$$
\sum_{\Pi \in \mathcal{B}(\lambda)} a^{o(\Pi)}q^{|\Pi|}= \det
M_{\lambda}(aq;q).
$$
\end{proposition}
\begin{proof} From (\ref{genoverpartitions}) we have that
the limit when $x$ runs to infinity of the number of paths from
$(0,0)$ to $(x,k)$ is $(-aq)_k/(q)_k$. Using Lemma 1 of
\cite{Li} we have that $ \det M_\lambda(aq;q)$ is the limit when $x$
goes to infinity of the generating function for $\ell(\lambda)$
nonintersecting paths going from $(0,\lambda_1+i-1-\lambda_i)$ to
$(x,\lambda_1+i-1)$. Thanks to the bijection between paths and
overpartitions this is also the generating function for
overpartitions whose $i$th row overpartition has at most
$\lambda_{i}$ and at least $\lambda_{i+1}$ parts.
\end{proof}


\begin{proposition}\label{gen}
Let $\lambda$ be a partition. The weighted generating function for
plane overpartitions of shape $\lambda$ is given by
$$
\sum_{\Pi \in \mathcal{S}(\lambda)} a^{o(\Pi)}q^{|\Pi|}=
q^{|\lambda|}\det M_{\lambda}(a;q).
$$
\end{proposition}
\begin{proof}
The proof is the same as the proof of Proposition \ref{bounded}.
Here we just use (\ref{genoverpartitionsprim}) instead of
(\ref{genoverpartitions}).
\end{proof}

The determinant of $M_{\lambda}(a;q)$ is given by the following formula:

\begin{equation}\label{m}
\det M_{\lambda}(a;q)=q^{\sum_i
(i-1)\lambda_i}\prod_{(i,j)\in\lambda}\frac{1+aq^{c_{i,j}}}{1-q^{h_{i,j}}}.
\end{equation}
For the proof see for example pages 16--17 of \cite{GV1}  or Theorem 26 (3.11) of \cite{Kr1}.

Using (\ref{m})  we obtain the product formulas for the generating functions from Propositions \ref{bounded} and \ref{gen}. Those are the  hook--content
formulas given in Theorems \ref{T1} and \ref{T2}.

{\bf Remark 1.} We give a bijective proof of formula (\ref{m}). Indeed Krattenthaler  \cite{Kr0} showed that the weighted generating function of
super semistandard Young tableaux of shape $\lambda$ is
$$
q^{\sum_i(i-1)\lambda_i}\prod_{(i,j)\in\lambda}\frac{1+aq^{c_{i,j}}}{1-q^{h_{i,j}}},
$$
where each tableau is weighted by $a^{o(T)}q^{|T|}$. A super semistandard tableau \cite{Kr0} of shape $\lambda$ is a filling of  cells of the Ferrers diagram of $\lambda$ with entries
from the ordered alphabet $1 < 2 < 3 < \ldots < \bar{1} < \bar{2} < \bar{3} < \ldots$ such that
\begin{itemize} 
\item the nonoverlined entries form a column-strict reverse plane partition of some shape $\nu$,
where $\nu$ is a partition contained in $\lambda$,
\item the overlined entries form a row-strict reverse plane partition of shape $\lambda / \nu$.
\end{itemize}

We give here a bijection from super semistandard Young tableaux to plane overpartitions.
We starting with a super semistandard Young tableau where the nonoverlined (resp. overlined)
entries are less than or equal to $k$ (resp. $\ell$).
First, change all the entries equal to $a$ by $k+1-a$ and all the entries equal
to $\bar{a}$ to $\overline{\ell+1-a}$. 

For example,  starting with the super semistandard Young tableau with $k=5$ and $\ell=4$
$$
\begin{array}{l}
1  3  4  5  5  \bar{2} \bar{4}\\
2  4  5  \bar{1} \bar{2} \bar{3} \bar{4}\\
3  5  \bar{2} \bar{3} \bar{4} \\
4  \bar{1} \bar{3}\\
\bar{3}
\end{array}
\quad \textrm{we obtain} \quad  
\begin{array}{l}
5  3  2  1  1  \bar{3} \bar{1}\\
4  2  1  \bar{4} \bar{3} \bar{2} \bar{1}\\
3  1  \bar{3} \bar{2} \bar{1} \\
2  \bar{4} \bar{2}\\
\bar{2}
\end{array}.
$$

Now we use the order $0<\bar{1}<1<\bar{2}<2<\bar{3}<3<\ldots $ and we suppose
that the cells outside of the shape $\lambda$ are filled with 0.
At first all the nonoverlined entries are active. While there is an active
part, we choose the smallest active part. If there are more than one, we choose the rightmost one. We swap it with its east or south neighbor, choosing the larger of the two. If they are equal then we choose the east one. We proceed in this way until this part is greater than or equal to both of its neighbors. When we reach this, 
we declare the part inactive.

Continuing with the preceding example and applying the algorithm to the smallest (and rightmost) active part until it becomes inactive we obtain:
$$
\begin{array}{l}
{\bf 5  3  2  1}  {\bf 1}  \bar{3} \bar{1}\\
{\bf 4  2  1}  \bar{4} \bar{3} \bar{2} \bar{1}\\
{\bf 3  1}  \bar{3} \bar{2} \bar{1} \\
{\bf 2}  \bar{4} \bar{2}\\
\bar{2}
\end{array} \ \ \ \begin{array}{l}
{\bf 5  3  2  1}   \bar{3} {\bf 1} \bar{1}\\
{\bf 4  2  1}  \bar{4} \bar{3} \bar{2}  \bar{1}\\
{\bf 3  1}  \bar{3} \bar{2} \bar{1} \\
{\bf 2}  \bar{4} \bar{2}\\
\bar{2}
\end{array}\ \ \ \begin{array}{l}
{\bf 5  3  2  1}   \bar{3} \bar{2} \bar{1}\\
{\bf 4  2  1}  \bar{4} \bar{3} {\bf 1}  \bar{1}\\
{\bf 3  1}  \bar{3} \bar{2} \bar{1} \\
{\bf 2}  \bar{4} \bar{2}\\
\bar{2}
\end{array} 
$$

Then we move the rest of the active parts.
$$
\begin{array}{l}
{\bf 5  3  2 } \bar{4}  \bar{3} \bar{2} \bar{1}\\
{\bf 4  2  1}   \bar{3} 1 { 1}  \bar{1}\\
{\bf 3  1}  \bar{3} \bar{2} \bar{1} \\
{\bf 2}  \bar{4} \bar{2}\\
\bar{2}
\end{array} \ \ \ \begin{array}{l}
{\bf 5  3  2}  \bar{4} \bar{3} \bar{2} \bar{1}\\
{\bf 4  2}   \bar{3}  \bar{2} 1 1 \bar{1}\\
{\bf 3  1}  \bar{3} { 1}  \bar{1} \\
{\bf 2}  \bar{4} \bar{2}\\
\bar{2}
\end{array} \ \ \ \begin{array}{l}
{\bf 5  3  2}  \bar{4} \bar{3} \bar{2} \bar{1}\\
{\bf 4  2}   \bar{3}  \bar{2} 1 1 \bar{1}\\
{\bf 3  }  \bar{4} \bar{3} { 1}  \bar{1} \\
{\bf 2}   \bar{2} {\bf 1} \\
\bar{2}
\end{array} \ \ \ \begin{array}{l}
{\bf 5  3  2}   \bar{4} \bar{3}  \bar{2} \bar{1}\\
{\bf 4  2}   \bar{3}  \bar{2} 1 1 \bar{1}\\
{\bf 3 }  \bar{4} \bar{3} { 1}  \bar{1} \\
{\bf 2}   \bar{2} { 1} \\
\bar{2}
\end{array} \ \ \ \begin{array}{l}
{\bf 5  3  } \bar{4} \bar{3} 2 \bar{2} \bar{1}\\
{\bf 4  2  } \bar{3}  \bar{2} 1 1 \bar{1}\\
{\bf 3  } \bar{4} \bar{3}  { 1}  \bar{1} \\
{\bf 2  } \bar{2} {1} \\
\bar{2}
\end{array} 
$$
$$
\begin{array}{l}
{\bf 5  3  } \bar{4} \bar{3} 2 \bar{2} \bar{1}\\
{\bf 4   } \bar{4} \bar{3}  \bar{2} 1 1 \bar{1}\\
{\bf 3  }  \bar{3} 2  { 1}  \bar{1} \\
{\bf 2  } \bar{2} {1} \\
\bar{2}
\end{array} \ \ \  \begin{array}{l}
{\bf 5  3  } \bar{4} \bar{3} 2 \bar{2} \bar{1}\\
{\bf 4   } \bar{4} \bar{3}  \bar{2} 1 1 \bar{1}\\
{\bf 3  }  \bar{3} 2  { 1}  \bar{1} \\
 \bar{2} 2 {1} \\
\bar{2}
\end{array} \ \ \ \begin{array}{l}
{\bf 5   } \bar{4} 3 \bar{3} 2 \bar{2} \bar{1}\\
{\bf 4   } \bar{4} \bar{3}  \bar{2} 1 1 \bar{1}\\
{\bf 3  }  \bar{3} 2  { 1}  \bar{1} \\
 \bar{2} 2 {1} \\
\bar{2}
\end{array} \ \ \ \begin{array}{l}
{\bf 5   } \bar{4} 3 \bar{3} 2 \bar{2} \bar{1}\\
{\bf 4   } \bar{4} \bar{3}  \bar{2} 1 1 \bar{1}\\
{ 3  }  \bar{3} 2  { 1}  \bar{1} \\
 \bar{2} 2 {1} \\
\bar{2}
\end{array} \ \ \ \begin{array}{l}
{\bf 5   } \bar{4} 3 \bar{3} 2 \bar{2} \bar{1}\\
{ 4   } \bar{4} \bar{3}  \bar{2} 1 1 \bar{1}\\
{ 3  }  \bar{3} 2  { 1}  \bar{1} \\
 \bar{2} 2 {1} \\
\bar{2}
\end{array}
$$

%

{\bf Remark 2.} Stanley's hook--content formula states that the generating function
of semistandard Young tableaux of shape $\lambda$ where the
entries are less than or equal to $n$ is
\begin{equation}\label{Shc}
q^{\sum(i-1) \lambda_i}\prod_{(i,j)\in\lambda}\frac{1-q^{n+c_{i,j}}}{1-q^{h_{i,j}}}.
\end{equation}
See Theorem 7.21.2 of \cite{St2}. Formula (\ref{m}) is equivalent to  Stanley's hook formula by Examples I.2.5 and I.3.3 of \cite{Mac}.  Again we can give a bijective argument. 

Formula (\ref{Shc}) is the weighted generating function of super semistandard
tableaux of shape $\lambda$ with $a=-q^n$. Start with a super semistandard Young tableau $T$ of shape $\lambda$ 
and add $n$ to all the overlined
entries.  Now transform
the tableau into a reverse plane overpartition with the above algorithm  
using the order $1>\bar{1}>2>\bar{2}>3>\ldots $. This shows that super semistandard Young tableaux $T$ of shape $\lambda$ with weight $(-q^n)^{o(T)}q^{|T|}$  are in bijection with
reverse plane overpartitions $\Pi$ of the same shape where the overlined entries are greater than $n$
with  weight $(-1)^{o(\Pi)}q^{|\Pi|}$. 

Now we define a sign reversing involution on these reverse plane overpartitions. Given such a  reverse plane overpartition, if there is at least one part greater than $n$, we choose the uppermost and rightmost part greater
than $n$. If this part is overlined, then we take off the overline, otherwise we overline the part.
Note that in this case the parity of the number of overlined parts is changed and therefore
the weight of the reverse plane overpartition is multiplied by $-1$.
If no such part exists, the given reverse plane partition  is a semistandard Young tableau of shape $\lambda$
where the entries are less than or equal to $n$.\\

Now, we give a generating formula for plane overpartitions with at most
$r$ rows and $c$ columns.


\begin{proposition}\label{suma}
The weighted generating function for plane overpartitions with
at most $r$ rows and $c$ columns is given by
$$
\sum_{c\ge \lambda_1\ge \ldots \ge \lambda_{(r-1)/2}\ge 0} \det
M_{(c, \lambda_1, \lambda_1,\ldots , \lambda_{(r-1)/2},
\lambda_{(r-1)/2})}(aq;q)
$$
if $r$ is odd, and by
$$
\sum_{c\ge \lambda_1\ge \ldots \ge \lambda_{r/2}\ge 0} \det
M_{(\lambda_1, \lambda_1,\ldots , \lambda_{r/2},
\lambda_{r/2})}(aq;q)
$$
if $r$ is even.

In particular, the weighted generating function for all
overpartitions is:
$$
\sum_{\lambda_1\ge \ldots \ge \lambda_{k}} \det M_{(\lambda_1,
\lambda_1,\ldots , \lambda_{k}, \lambda_{k})}(aq;q).
$$
\end{proposition}
\begin{proof}
This is a direct consequence of Proposition \ref{bounded}.
\end{proof}

We will use this result to get another ``symmetric
function'' proof of the shifted MacMahon's formula 
(\cite{FW,V1,V2}):
$$
\sum_{\substack{\Pi \text { is a
plane}\\\text{overpartition}}}q^{|\Pi|}=\prod_{n=1}^\infty
\left(\frac{1+q^n}{1-q^n}\right)^n.
$$

\subsection{The weighted shifted MacMahon formula}


In this section we give the weighted generalization of the shifted MacMahon formula.
We  use a symmetric function identity to compute the last sum in Proposition \ref{suma} since  $\det M_\lambda(a;q)$, as we will soon see,  has an interpretation in terms of symmetric functions. 

A symmetric function of  an alphabet (set of indeterminates, also called letters) $\mathbb A$ is a function of letters which is invariant under any permutation of $\mathbb A$.  Recall three standard bases for
the algebra of symmetric functions: Schur functions $s$, complete
symmetric functions $h$ and elementary symmetric functions $e$ (see \cite{Mac}). The latter two are conveniently given by their generating functions:
$$
H_t(\mathbb A )=\prod_{a\in A}\frac{1}{1-ta}\quad \text{and} \quad E_t(\mathbb A )=\prod_{a\in A}(1+ta).
$$

Algebraic operations  on alphabets, such as addition, subtraction and multiplication, can be defined  using $\lambda$-rings framework, see Chapters I and  II of  \cite {La}. In this framework symmetric functions are seen as ring operators. The addition of two alphabets $\mathbb A$ and $\mathbb B$ is defined naturally as their disjoint union and is denoted by $\mathbb A + \mathbb B$.  Obviously,
$$
H_t(\mathbb A+\mathbb B)=H_t(\mathbb A)H_t(\mathbb B).
$$
Subtraction and multiplication  are defined by
$$
H_t(\mathbb A -\mathbb B)=\frac{H_t(\mathbb A)}{H_t(\mathbb B)}, \quad \quad H_t(c\mathbb A)=(H_t(\mathbb A))^{c}, \quad c\in \mathbb C.
$$
Analogous relations hold for elementary symmetric functions since $E_t(\mathbb A)=H_{-t}(-\mathbb A)$.

A specialization is an algebra homomorphism between the algebra of
symmetric functions and $\mathbb{C}$.   If
$\rho$ is a specialization  and $f$ is a symmetric function we denote its image by $f|_{\rho}$.

%

Let $\rho(a)$ be a specialization given by
$$
h_n|_{\rho(a)}=\frac{(-a)_{n}}{(q)_{n}}.
$$
Since, for
$\lambda={(\lambda_1,\lambda_2,\dots,\lambda_k)}$
$$
s_{\lambda}=\det(h_{\lambda_i-i+j})
$$
we have from Definition \ref{DefM} that
\begin{equation}\label{sym}
s_{\lambda}|_{\rho(a)}=\det M_\lambda(a;q).
\end{equation}

In the $\lambda$--ring framework, the $q$-binomial theorem (see (2.21) of \cite{A})
$$
\sum_{n=0}^\infty \frac{(-a)_n}{(q)_n}t^n=\frac{(-at)_\infty}{(t)_\infty}
$$
shows that the specialization $\rho(a)$ is equivalent to considering symmetric functions in the difference of two alphabets $1+q+q^2+\cdots$ and $-a-aq-aq^2-\cdots$. Thus,
\begin{equation*}
s_{\lambda}|_{\rho(a)}=s_\lambda((1+q+q^2+\cdots)-(-a-aq-aq^2-aq^3-\cdots)).
\end{equation*}

The weighted shifted MacMahon formula
can be obtained from (\ref{sym}) and Proposition
\ref{suma}. We have 
$$
\sum_{\substack{\Pi \text { is a
plane}\\\text{overpartition}}}a^{o(\pi)}q^{|\Pi|}=\sum_{\lambda_1\geq
\lambda_2\geq\cdots\geq
\lambda_k}s_{(\lambda_1,\lambda_1,\lambda_2,\lambda_2,\dots,\lambda_k,\lambda_k)}|_{\rho(aq)}=\sum_{\lambda'
\text{ even}}s_\lambda|_{\rho(aq)},
$$
where $\lambda'$ is the transpose of $\lambda$ and a partition is
even if it has even parts.


By Ex. 10 b) on p. 79 of \cite{Mac} we have
$$
\sum_{\lambda' \text{
even}}s_\lambda(\mathbb A)t^{|\lambda|/2}=H_t(e_2(\mathbb A)).
$$
If instead of $\mathbb A$ we insert a difference of alphabets $\mathbb A=x_1+x_2+\dots$ and $\mathbb B=y_1+y_2+\dots$ then we obtain the following product formula:
\begin{eqnarray*}
\sum_{\lambda' \text{
even}}s_\lambda(\mathbb A - \mathbb B)t^{|\lambda|/2}&=&H_t(e_2(\mathbb A-\mathbb B))\\
&=&H_t\Big(\sum_{1  \leq i  < j} x_ix_j+\sum_{1\leq i\leq j} y_iy_j-\sum_{1\leq i,j}x_iy_j\Big)\\
&=&\prod_{1\leq i <j}\frac{1}{1-x_ix_jt}\prod_{1\leq i\leq j}\frac{1}{1-y_iy_jt}\prod_{1 \leq i ,j} (1-x_iy_jt ).
\end{eqnarray*}
This gives us the weighted shifted MacMahon formula:
\begin{proposition} The weighted generating formula for plane overpartitions is
$$
\sum_{\substack{\Pi \text { is a
plane}\\\text{overpartition}}}a^{o(\Pi)}q^{|\Pi|}=
\prod_{i=1}^\infty \frac{(1+aq^i)^{i}}{(1-q^i)^{\lceil
i/2\rceil}(1-a^2q^i)^{\lfloor i/2\rfloor}}.
$$
\label{propweighted}
\end{proposition}
\begin{proof}
Substituting $x_i=q^{i-1}$ and $y_i=-aq^i$ in the preceding product formula, we get
$$
\prod_{0\leq i <j}\frac{1}{1-q^{i+j}}\prod_{1 \leq i \leq j}\frac{1}{1-a^2q^{i+j}}\prod_{1 \leq i,j}(1+aq^{i+j-1}).
$$
\end{proof}

{\bf Remark.} This proof was suggested to the authors by one of the anonymous referees.


\subsection{Reverse plane overpartitions}


In this section, we construct a bijection between the set of all reverse plane overpartitions and sets of nonintersecting paths whose endpoints are not fixed. We use this bijection and Stembridge's results \cite{St} to obtain a Pfaffian formula for the generating function for reverse plane overpartitions of a given shape. Evaluating the Pfaffian we obtain the hook formula for reverse plane overpartitions due to Okada \cite{O}. Let $\mathcal{S}^R(\lambda)$ be the set of all reverse plane partitions of shape $\lambda$.

\begin{theorem} \label{Okada} 
The generating function for reverse plane overpartitions of shape $\lambda$ is
$$
\sum_{\Pi \in \mathcal{S}^R(\lambda)}q^{|\Pi|}=\prod_{(i,j)\in
\lambda}\frac{1+q^{h_{i,j}}}{1-q^{h_{i,j}}}.
$$
\end{theorem}

We construct a weight preserving bijection between reverse plane overpartitions and sets of nonintersecting paths on a triangular lattice in a similar fashion as in Section \ref{3.1}. The lattice consists of East, North and North--East edges. East edges have weight $1$, North edges on (vertical) level $i$ have weight $q^{i+1}$ and North--East edges joining vertical levels $i$ and $i+1$ have weight $q^{i+1}$. The weight of a set of nonintersecting paths $p$ is the product of the weights of their edges and is denoted by $w(p)$. Let $\Pi$ be a reverse plane overpartition whose positive entries form a skew shape $\lambda / \mu$ and let
$\ell=\ell(\lambda)$. 
Then $\Pi$
 can be represented by a set of $n$ nonintersecting lattice paths such that
\begin{itemize}
\item the departure points are $(0,\mu_i+\ell-i)$ and
\item the arrivals points are $(x,\lambda_i+\ell-i)$,
\end{itemize}
for a large enough $x$ and $i=1,\dots,\ell$.
For example let $x=8$, $\lambda=(5,4,2,2)$ and $\mu=(2,1)$.
Figure
\ref{Fig44} shows  the corresponding set of
nonintersecting paths for  the reverse plane overpartition
of shape $\lambda/\mu$
$$
\begin{array}{llllll}
&&{3}&4&4&\\
&3&{4}&\bar{4}&\\
1&\bar{3}&&&&\\
3&\bar{3}&&&
\end{array}.
$$

\begin{figure} [htp!]
\centering
\includegraphics[height=4cm]{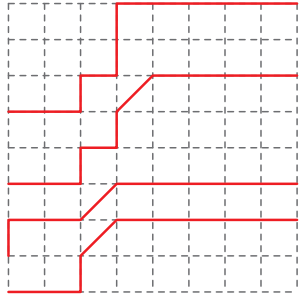}
\caption{Nonintersecting paths and reverse plane overpartitions}
\label{Fig44}
\end{figure}

This implies that all reverse plane overpartitions of shape $\lambda=(\lambda_1,\ldots ,\lambda_\ell)$ 
can be represented  by nonintersecting lattice
paths such that
\begin{itemize}
\item the departure points are an $\ell$--element subset of $\{(0,i)\ | i\geq 0\}$ and
\item the arrivals points are $(x,\lambda_i+\ell-i),$
\end{itemize}
with $x\rightarrow \infty$.\\

Now, for $r_1>r_2>\dots>r_\ell \geq0$ we define
$$
W(r_1,r_2,\dots,r_\ell)=\lim_{x \to \infty}\sum_{p \in P(x;r_1,\dots,r_\ell)}w(p),
$$
where $P(x;r_1,r_2,\dots,r_\ell)$ is the set of all nonintersecting paths joining an $\ell$-element subset of $\{(0,i) |  i\geq 0\}$ with $\{(x,r_1),\dots,(x,r_\ell)\}$. Note that for $r_1>r_2>\dots>r_\ell>0$ we have
\begin{equation}\label{rekW}
W(r_1,\dots,r_{\ell},0)=W(r_1-1,\dots,r_\ell-1).
\end{equation}

By Stembridge's Pfaffian formula for the sum of the weights of nonintersecting paths where departure points are not fixed (Theorem 3.1 of \cite{St}) we obtain
$$
W(r_1,r_2,\dots,r_\ell)=\text{Pf}(D),
$$
where if $\ell$ is even $D$ is the $\ell \times \ell$ skew--symmetric matrix defined by $D_{i,j}=W(r_i,r_j)$ for $1\leq i<j\leq \ell$ and if $\ell$ is odd $D$ is the $(\ell+1) \times (\ell+1)$ skew--symmetric matrix defined by $D_{i,j}=W(r_i,r_j)$  for $1\leq i <j\leq \ell$ and $D_{i,\ell+1}=W(r_i)$ for $1\leq i\leq \ell$.

\begin{lemma} Let $r>s\geq0$. Then
\begin{eqnarray}
W(s)&=&\frac{(-q)_s}{(q)_s},\label{forzaf1}\\
W(r,s)
&=&\frac{(-q)_{r}}{(q)_{r}}\cdot\frac{(-q)_s}{(q)_s}\cdot\frac{1-q^{r-s}}{1+q^{r-s}}.\label{forzaf2}
\end{eqnarray}
\end{lemma}
\begin{proof}
From Lemma \ref{lemaop} we have that the generating function for overpartitions with at most $n$ parts is $M(n)=(-q)_n/(q)_n$ and 
 the generating function for overpartitions with exactly $n$ parts is $P(n)=q^n(-1)_n/(q)_n$.  This implies that
\begin{equation}\label{pz}
\sum_{i=0}^{n}P(i)=M(n).
\end{equation}
Moreover,
$$
W(s)=M(s)=\frac{(-q)_s}{(q)_s}\;\;\;\;\;\text{and}\;\;\;\;\;W(r,0)=M(r-1)=\frac{(-q)_{r-1}}{(q)_{r-1}}.
$$

We prove (\ref{forzaf2}) by induction on $s$. The formula for the base case $s=0$ holds by the above. So, we assume $s\geq1$.

By Lindstr\"om's determinantal formula (Lemma 1 of \cite{Li}) we have that
$$
W(r,s)=\sum_{i=0}^{s}\sum_{j=i+1}^{r}\left(P(r-j)P(s-i)-P(r-i)P(s-j)\right).
$$

%
Summing over $j$ and using (\ref{pz}) we obtain
$$
W(r,s)=\sum_{i=0}^s \left(P(s-i)M(r-i-1)-P(r-i)M(s-1-i)\right).
$$
Then
$$
W(r,s)=W(r-1,s-1)+P(s)M(r-1)-P(r)M(s-1).
$$
It is enough to prove that
\begin{equation}\label{kolic}
\frac{P(s)M(r-1)-P(r)M(s-1)}{W(r-1,s-1)}=\frac{2(q^{r}+q^s)}{(1-q^{r})(1-q^s)}
\end{equation}
and (\ref{forzaf2}) follows by induction.
Now,
\begin{eqnarray*}
&&P(s)M(r-1)-P(r)M(s-1)=\\
&&=q^s\frac{(-1)_s}{(q)_s}\cdot\frac{(-q)_{r-1}}{(q)_{r-1}}-q^{r}\frac{(-1)_{r}}{(q)_{r}}\cdot\frac{(-q)_{s-1}}{(q)_{s-1}}\\
&&=\frac{(-q)_{r-1}}{(q)_{r-1}}\cdot\frac{(-q)_{s-1}}{(q)_{s-1}}\cdot\frac{1+q^{r-s}}{1-q^{r-s}}\cdot\frac{2(q^{r}+q^s)}{(1-q^{r})(1-q^s)}.
\end{eqnarray*}
Using the inductive hypothesis for $W(r-1,s-1)$ we obtain (\ref{kolic}).

\end{proof}

Let $F_\lambda$ be the generating function for reverse plane overpartitions of shape $\lambda$. Then,
using the bijection we have constructed, we obtain
$$
F_\lambda=W(\lambda_1+\ell-1,\lambda_2+\ell-2,\dots,\lambda_\ell),
$$
which, after applying Stembridge's result, gives us a Pfaffian formula. This Pfaffian formula can be expressed as a product after the following observations.

Let M be $2n \times 2n$ skew--symmetric matrix. One of definitions of the Pfaffian is the following:
$$
\text{Pf}(M)= \sum_{\pi=(i_1,j_1)\ldots (i_n,j_n)} \text{sgn} (\pi)
M_{i_1,j_1}M_{i_2,j_2}\cdots M_{i_n,j_n},
$$
where the sum is over all of perfect matchings (or fixed point free involutions) of $[2n]$.
Also, for $r>s>0$,
\begin{eqnarray*}
W(s)&=&\frac{1+q^s}{1-q^s} W(s-1),\\
{\rm and}\ \ W(r,s)&=&\frac{1+q^r}{1-q^r} \cdot \frac{1+q^s}{1-q^s}W(r-1,s-1).
\end{eqnarray*}
Then
$$
F_\lambda=\prod_{j=1}^\ell\frac{1+q^{h_{j,1}}}{1-q^{h_{j,1}}}\cdot F_{\bar{\lambda}},
$$
where $\bar{\lambda}=(\lambda_1-1,\dots,\lambda_\ell-1)$ if $\lambda_\ell>1$ and $\bar{\lambda}=(\lambda_1-1,\dots,\lambda_{\ell-1}-1)$ if $\lambda_\ell=1$ (see (\ref{rekW}) in this case). Inductively we obtain Theorem \ref{Okada}.

\section{Domino tilings}

In \cite{V1} a measure on (diagonally) strict plane partitions was studied.
Strict plane partitions are plane partitions were all diagonals are
strict partitions, i.e. strictly decreasing sequences. They can also
be seen as plane overpartitions where all overlines are deleted.
There are $2^{k(\Pi)}$ different plane overpartitions corresponding
to the same strict plane partition $\Pi$, where $k(\Pi)$ is the number of 
connected components of $\Pi$. 

Alternatively, a strict plane partition can be seen as a subset of
$\mathbb{N}\times\mathbb{Z}$ consisting of points $(t,x)$ where $x$
is a part of the diagonal partition indexed by $t$. See Figure
\ref{Fig5}. We call this set the 2-dimensional diagram of that
strict plane partition. The connected components are connected sets
(no holes) on the same horizontal line. The 2-dimensional diagram of
a plane overpartition is the 2-dimensional diagram of its
corresponding strict plane partition.
\begin{figure} [htp!]
\centering
\includegraphics[height=5.5cm]{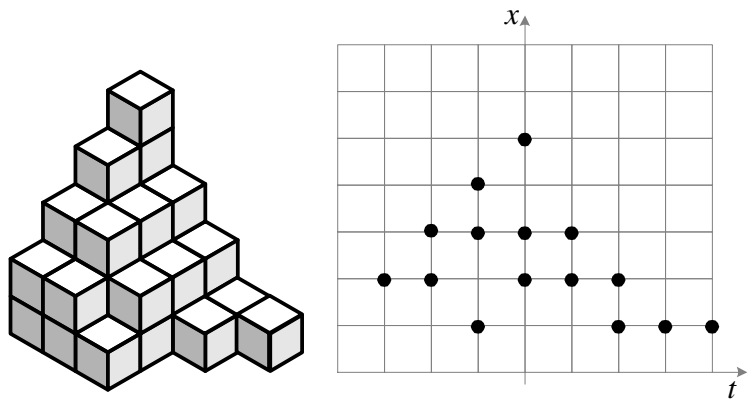}
\caption{A strict plane partition and its corresponding 2-dimensional diagram } \label{Fig5}
\end{figure}

The measure studied in \cite{V1} assigns to each strict plane
partition a weight equal to $2^{k(\Pi)}q^{|\Pi|}$. The limit shape
of this measure is given in terms of the Ronkin function of the
polynomial $P(z,w)=-1+z+w+zw$ and it is parameterized on the domain
representing half of the amoeba of this polynomial. This polynomial
is also related to plane tilings with dominoes. This, as well as
some other features like similarities in correlation kernels
\cite{J,V1} suggested that a connection between this measure and
domino tilings is likely to exist.

Alternatively, one can see this measure as a uniform measure on
plane overpartitions, i.e. each plane overpartition $\Pi$ has a
probability proportional to $q^{|\Pi|}$. In Section \ref{nonint} we
have constructed a bijection between plane overpartitions and
sets of nonintersecting paths. See Figure \ref{Fig6}. This figure is obtained  from Figure \ref{Fig4} by a rotation. Note that the paths are incident with all points in the corresponding 2-dimensional diagram.   The paths consist of edges of three different kinds: horizontal (joining
$(t,x)$ and $(t+1,x)$), vertical (joining $(t,x+1)$ and $(t,x)$) and
diagonal (joining $(t,x+1)$ and $(t+1,x)$). There is a standard way
to construct a tiling with dominoes using these paths(see for
example \cite{J}). We explain the process below. An example is given on Figure \ref{Fig6}.
\begin{figure} [htp!]
\centering
\includegraphics[height=6.5cm]{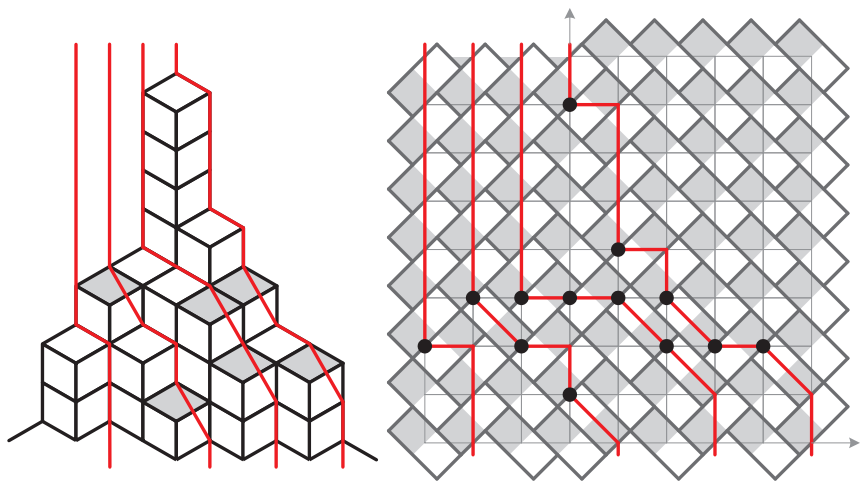}
\caption{A plane overpartition and its corresponding domino tiling}
\label{Fig6}
\end{figure}

We start from $\mathbb{R}^2$ and color it in a chessboard fashion
such that $(1/4,1/4)$, $(-1/4,3/4)$, $(1/4,5/4)$ and $(3/4,3/4)$ are
vertices of a white square. So, the axes of this infinite chessboard
form angles of $45$ and $135$ degrees with the axes of $\mathbb{R}^2$. A
domino placed on this infinite chessboard can be one of the four
types: $(1,1)$, $(-1,-1)$, $(-1,1)$ or $(1,-1)$, where we say that a
domino is of type $(x,y)$ if $(x,y)$ is the vector parallel to the
vector whose starting, respectively end point is the center of the
white, respectively black square of that domino.

Now, take a plane overpartition and represent it on this chessboard by its corresponding set of nonintersecting paths.  We cover each
edge by a domino that satisfies that the endpoints of that edge are
midpoints of sides of the black and white square of that domino.
In that way, we obtain a tiling of a part of the plane with dominoes of
three types: $(1,1)$, $(-1,-1)$ and $(1,-1)$. More precisely,
horizontal edges correspond to $(1,1)$ dominoes, vertical to
$(-1,-1)$ and diagonal to $(1,-1)$. To tile the whole plane we fill
the rest of it by dominoes of the fourth, $(-1,1)$ type. See Figure
\ref{Fig6} for an illustration.

In this way, we have established a correspondence between plane
overpartitions and plane tilings with dominoes. We now give some of
the properties of this correspondence. First, we describe how a
tiling changes  when we add or remove an overline or we
add or remove a box from a plane overpartition. We require that when
we add/remove an overline or a box we obtain a plane overpartition
again.

In terms of the 2-dimensional diagram of a plane partition, adding an overline
can occur at all places where $(t,x)$ is in the diagram and
$(t+1,x)$ is not. Adding an overline at $(t,x)$ means that a pair of horizontal and vertical
edges, $((t,x),(t+1,x))$ and $((t+1,x),(t+1,x-1))$, is replaced by one diagonal edge, $((t,x),(t+1,x-1))$. This means that the new
tiling differs from the old one by replacing a pair of $(1,1)$ and
$(-1,-1)$ dominoes by a pair of $(1,-1)$ and $(-1,1)$ dominoes. See
Figure \ref{Fig7}.
\begin{figure} [htp!]
\centering
\includegraphics[height=3.2cm]{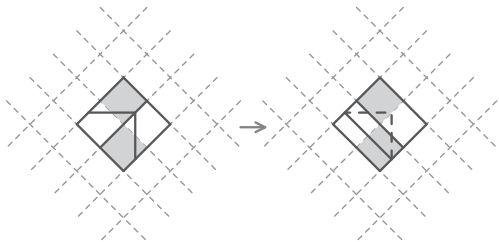}
\caption{Adding an overline} \label{Fig7}
\end{figure}
Removing an overline is the inverse of adding an overline.

We now explain the operation of removing a box.  Observe that if a box can be removed from a plane overpartition then the corresponding part is overlined or it can be overlined to obtain a plane
overpartition again. So, it is enough to consider how the tiling
changes when we remove an overlined box since we have
already considered the case of adding or removing an overline. If we
remove an overlined box we change a diagonal edge to a
pair of vertical and horizontal edges. If the box was represented by $(t,x)$ in the 2-dimensional diagram then the edge $((t,x),(t+1,x-1))$ is replaced by the pair of $((t,x),(t,x-1))$ and $((t,x-1),(t+1,x-1))$. This means that the new
tiling differs from the old one by replacing a pair of $(1,-1)$ and
$(-1,1)$ dominoes by a pair of $(1,1)$ and $(-1,-1)$ dominoes. See
Figure \ref{Fig8}. 
\begin{figure} [htp!]
\centering
\includegraphics[height=3.2cm]{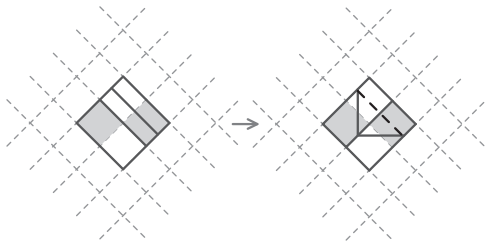}
\caption{Removing a box} \label{Fig8}
\end{figure}

All four operations are described by a swap of  a pair of adjacent $(1,1)$ and $(-1,-1)$ dominoes and a pair of adjacent of $(1,-1)$ and $(-1,1)$ dominoes.

We conclude this section by the observation that plane overpartitions
of a given shape $\lambda$ and whose parts are bounded by $n$ are in
bijection with domino tilings of the rectangle
$[-\ell(\lambda)+1,\lambda_1]\times [0,n]$ with certain boundary
conditions. These conditions are imposed by the fact that outside of
this rectangle nonintersecting paths are just straight lines. We
describe the boundary conditions precisely in the proposition below.
\begin{proposition}
The set $\mathcal{S}(\lambda) \cap \mathcal{L}(n)$ of all plane
overpartitions of shape $\lambda$ and whose largest part is at most
n is in bijection with plane tilings with dominoes where a point
$(t,x) \in \mathbb{Z} \times \mathbb{R}$ is covered by a domino of
type $(-1,-1)$ if
\begin{itemize}
\item $t\leq -\ell(\lambda),$
\item $-\ell(\lambda)<t\leq 0$ and $x\geq n,$
\item $t = \lambda_i-i+1$ for some $i$ and $x \leq0,$
\end{itemize}
and a domino of type $(-1,1)$ if
\begin{itemize}
\item $t>\lambda_1,$
\item $0<t\leq \lambda_1$ and $x\geq n+1/2,$
\item $t \neq \lambda_i-i+1$ for all $i$ and $x \leq -1/2.$
\end{itemize}
\end{proposition}
The boundary conditions for  $\lambda=(5,4,3,1)$ and $n=7$ are shown in  Figure \ref{Fig9}. The example from Figure \ref{Fig6} satisfies these boundary conditions.
\begin{figure} [htp!]
\centering
\includegraphics[height=7cm]{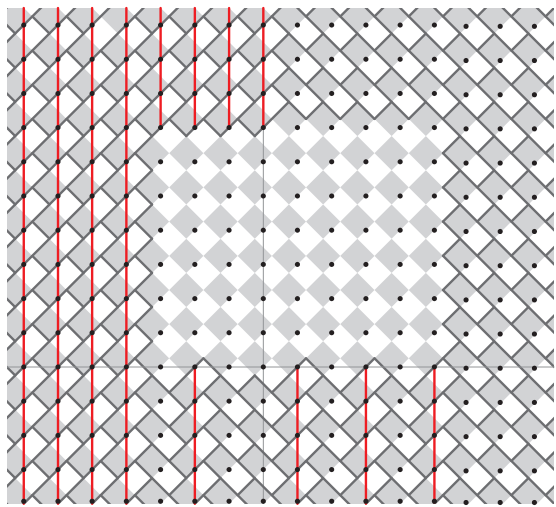}
\caption{Boundary conditions} \label{Fig9}
\end{figure}

\section{Robinson-Schensted-Knuth (RSK) type algorithm for plane overpartitions}
In this section we are going to give a bijection between certain matrices
and pairs of plane overpartitions of the same shape. This
bijection is obtained by an algorithm similar to the algorithm RS2 of Berele and Remmel
\cite{berele:remmel:1985} which gives a bijection between matrices
and pairs of $(k,\ell)$-semistandard tableaux. These tableaux are
super semistandard tableaux, defined in Section 3,
where the nonoverlined (resp. overlined) parts
are less than or equal to $k$ (resp. $\ell$).

We then apply properties of this algorithm to enumerate plane
overpartitions, as done by Bender and Knuth \cite{BK} for plane
partitions.

\subsection{The RSK algorithm}
Let $M=\begin{pmatrix}A & B \\C & D\end{pmatrix}$ be a $2n\times 2n$
matrix, made of four $n\times n$ blocks $A,B,C$ and $D$. The blocks $A$ and $D$ are nonnegative integer matrices, and
$B$ and $C$ are $\left\{0,1\right\}$ matrices. We denote the set of all such matrices with $\mathcal{M}_n$. We represent a matrix in $\mathcal{M}_n$ by a sequence of pairs of numbers  $\begin{pmatrix}i\\j\end{pmatrix}$, $\begin{pmatrix}i\\\bar{j}\end{pmatrix}$, $\begin{pmatrix}\bar{i}\\j\end{pmatrix}$ and 
$\begin{pmatrix}\bar{i}\\\bar{j}\end{pmatrix}$.

The encoding of $M$ into pairs is made using the following rules:
\begin{itemize}
    \item for each nonzero entry $a_{ij}$ of $A$, we create $a_{ij}$ pairs $\begin{pmatrix}i\\j\end{pmatrix}$,
    \item for each nonzero entry $b_{ij}$ of $B$, we create one pair $\begin{pmatrix}i\\\bar{j}\end{pmatrix}$,
    \item for each nonzero entry $c_{ij}$ of $C$, we create one pair $\begin{pmatrix}\bar{i}\\j\end{pmatrix}$,
    \item for each nonzero entry $d_{ij}$ of $D$, we create $d_{ij}$ pairs $\begin{pmatrix}\bar{i}\\\bar{j}\end{pmatrix}$.
\end{itemize}




For example let $M=\begin{pmatrix}0 & 2 & 1 & 0 \\2 & 0 & 1 & 0 \\1 & 1 & 0 &
1 \\0 & 0 & 1 & 1\end{pmatrix}$. After encoding $M$, we obtain
\begin{center} \[\begin{pmatrix}1\\2\end{pmatrix},\begin{pmatrix}1\\2\end{pmatrix},\begin{pmatrix}1\\\bar{1}\end{pmatrix},\begin{pmatrix}2\\1\end{pmatrix},\begin{pmatrix}2\\1\end{pmatrix}, \begin{pmatrix}2\\\bar{1}\end{pmatrix},\begin{pmatrix}\bar{1}\\1\end{pmatrix},\begin{pmatrix}\bar{1}\\2\end{pmatrix},\begin{pmatrix}\bar{1}\\\bar{2}\end{pmatrix},\begin{pmatrix}\bar{2}\\\bar{1}\end{pmatrix}, \begin{pmatrix}\bar{2}\\\bar{2}\end{pmatrix}.\]
\end{center}

\vskip 0.5cm

From now on, we fix the order $\bar{1}<1<\bar{2}<2<\bar{3}<3<...$. We
sort the pairs to create a two--line array $L$ such that
\begin{itemize}
    \item the first line is a nonincreasing sequence,
    \item if two entries of the first line are equal and overlined
    (resp. nonoverlined) then the corresponding entries in the
    second line are in weakly increasing (resp. decreasing) order.
\end{itemize}


For the example above, after sorting, we obtain the two--line array 
$$
L=\begin{pmatrix}2,2,2,\bar{2},\bar{2},1,1,1,\bar{1},\bar{1},\bar{1}\\
1,1,\bar{1},\bar{1},\bar{2},2,2,\bar{1},1,\bar{2},2\end{pmatrix} .
$$


We now describe the
\emph{insertion algorithm}. It is based on an algorithm proposed
by Knuth in \cite{K} and quite similar to the algorithm
RS2 of \cite{berele:remmel:1985}. 

We first explain how to insert a number $j$ into an overpartition $\lambda$ of length $\ell$, so that the sequence $\lambda$ is still an overpartition. If  $j$ can be inserted at the end of $\lambda$, then insert and stop. Otherwise, find the smallest part that can be replaced by $j$. Bump this part and insert $j$.  At the end we obtain a new overpartition that contains $j$ as a part and whose length is $\ell$ if we bumped a part or $\ell +1$ if $j$ was added at the end and no part was bumped.  For example if $\lambda=(4,3,3,\bar{3},2)$ then
\begin{itemize}
\item insert 5, obtain $(5,3,3,\bar{3},2)$ and bump 4,
\item insert 3, obtain $(4,3,3,3,2)$ and bump $\bar{3}$,
\item insert $\bar{3}$, obtain $(4,3,3,\bar{3},2)$ and bump $\bar{3}$,
\item insert 1, obtain $(4,3,3,\bar{3},2,1)$ and bump nothing.
\end{itemize}


To insert $j$ into a plane overpartition $P$ insert $j$ in the first row. If nothing is bumped then stop. Otherwise,  insert the bumped part in the second row. If something is bumped from the second row, insert it in the third and so on. Stop when nothing is bumped.
 For example when $\bar{3}$ is inserted in 
 $$
\begin{array}{lllll} 4&3&3&\bar{3}&2\\
3&\bar{3}&2&\\
1&&&&\\
 \end{array}
 \quad
 \textrm{we obtain}
 \quad 
 \begin{array}{lllll} 4&3&3&\bar{3}&2\\
3&\bar{3}&2&\\
\bar{3}&&&&\\
1&&&&\\
 \end{array}.
 $$


We define how to insert a pair
$\begin{pmatrix}i\\j\end{pmatrix}$ into a pair of plane
overpartitions of the same shape $(P,Q)$.
We first insert $j$ in $P$ with the insertion algorithm. If the insertion ends in column $c$
and row $r$ of $P$, then insert $i$ in column $c$ and row $r$ in
$Q$.
Finally going from the two--line array $L$ to pairs of plane
overpartitions of the same shape works as follows: start with two
empty plane overpartitions and insert each pair of $L$ going from
left to right. This is identical to the classical RSK
algorithm \cite{K}.

Continuing with the previous example and applying the insertion algorithm we get
$$
\begin{array}{ll}
P=\begin{array}{lll} 2&2&2\\
\bar{2}&1&1\\
\bar{2}&\bar{1}&\\
1&&\\
\bar{1}&&\\
\bar{1}&& \end{array}, \ \ \ \ \ \
Q=\begin{array}{lll} 2&2&2\\
\bar{2}&1&1\\
\bar{2}&\bar{1}&\\
1&&\\
\bar{1}&&\\
\bar{1}&& \end{array}\end{array}.
$$

Let $\mathcal{L}(n)$ be the set of all plane overpartitions  with the largest entry at most  $n$.

\begin{theorem}
There is a one-to-one correspondence between matrices $M\in\mathcal{M}_n$ and
pairs of plane overpartitions of the same shape $(P,Q) \in \mathcal{L}(n) \times \mathcal{L}(n)$.
This correspondence is such that~:
\begin{itemize}
\item $k$ appears in $P$ exactly $\sum_i a_{ik}+c_{ik}$ times, 
\item $\bar{k}$ appears in $P$ exactly $\sum_i b_{ik}+d_{ik}$ times, 
\item $k$ appears in $Q$ exactly $\sum_i a_{ik}+b_{ik}$ times,
 \item $\bar{k}$ appears in $Q$ exactly $\sum_i c_{ik}+d_{ik}$  times.
\end{itemize}
\end{theorem}
\begin{proof}
The proof is identical to the proof in the case of the RSK algorithm
\cite{K} or the RS2 algorithm \cite{berele:remmel:1985}.  Details
are given in \cite{sav}.
\end{proof}

\begin{theorem}
If the {insertion algorithm} produces $(P,Q)$ with input matrix
$M$, then the {insertion algorithm} produces $(Q,P)$ with input
matrix $M^T$. \label{insalg}
\end{theorem}
\begin{proof} The proof is again analogous to the one in \cite{K}.
Given a two--line array $
\begin{pmatrix}
u_1,\ldots ,u_N\\
v_1,\ldots ,v_N \end{pmatrix}$, we partition the pairs
$\begin{pmatrix}u_\ell\\ v_\ell\end{pmatrix}$ in classes such that 
$\begin{pmatrix}u_k\\ v_k\end{pmatrix}$
 and $\begin{pmatrix}u_m\\ v_m\end{pmatrix}$ are
in the same class if and only if~:
\begin{itemize}
\item $u_k\ge u_m$ and if $u_k=u_m$, then $u_m$ is overlined AND
\item $v_k\le v_m$ and if $v_k=v_m$, then $v_k$ is overlined.
\end{itemize}
Then one can sort each class so that the first entries of each pair
appear in nonincreasing order and then sort the classes so that
the first entries of the first pair of each class are in
nonincreasing order. For example if the two--line array is
\begin{center}
$
\begin{pmatrix}2,2,2,\bar{2},\bar{2},1,1,1,\bar{1},\bar{1},\bar{1}\\
1,1,\bar{1},\bar{1},\bar{2},2,2,\bar{1},1,\bar{2},2\end{pmatrix}
$,
\end{center}
we get the classes
$$
C_1=\left\{\begin{pmatrix}2\\1\end{pmatrix},\begin{pmatrix}\bar{2}\\\bar{2}\end{pmatrix},\begin{pmatrix}1\\2\end{pmatrix}\right\},\ \ 
C_2=\left\{\begin{pmatrix}2\\1\end{pmatrix},\begin{pmatrix}{1}\\{2}\end{pmatrix}\right\},\ \ 
$$
$$
C_3=\left\{
\begin{pmatrix}2\\\bar{1}\end{pmatrix},\begin{pmatrix}\bar{2}\\\bar{1}\end{pmatrix},\begin{pmatrix}1\\ \bar{1}\end{pmatrix},\begin{pmatrix}\bar{1}\\1\end{pmatrix},\begin{pmatrix}\bar{1}\\\bar{2}\end{pmatrix},\begin{pmatrix}\bar{1}\\2\end{pmatrix} \right\}.\ \ 
$$
If the classes are $C_1,\ldots ,C_d$ with
$$C_i=\left\{\begin{pmatrix}u_{i1}\\v_{i1}\end{pmatrix},\ldots ,\begin{pmatrix}u_{in_i}\\v_{in_i}\end{pmatrix}\right\}$$ then the
first row of $P$ is
$$v_{1n_1},\ldots ,v_{dn_d}$$ and the first row of $Q$ is
$$u_{11},\ldots ,u_{d1}.$$ Moreover one constructs the rest of $P$
and $Q$ using the pairs~:
$$
\bigcup_{i=1}^d
\bigcup_{j=1}^{n_i-1}\begin{pmatrix}u_{i,j+1}\\v_{ij}\end{pmatrix}.
$$
One can adapt the proof of Lemma 1 of \cite{K}  for a complete proof. 
This is done in the master thesis of the second author \cite{sav}. As
the two--line array corresponding to $M^T$ is obtained by
interchanging the two lines of the array and rearranging the
columns, the theorem follows.
\end{proof}
This implies that $M=M^T$ if and only if $P=Q$.
\begin{theorem} There is a one-to-one correspondence between
symmetric matrices $M \in \mathcal{M}_n$ and  plane overpartitions  $P\in \mathcal{L}(n)$. In this
correspondence:
\begin{itemize}
\item $k$ appears in $P$ exactly $\sum_i a_{ik}+c_{ik}$ times, and\item
$\bar{k}$ appears in $P$ exactly $\sum_i b_{ik}+d_{ik}$ times.
\end{itemize}
\label{symmetric}
\end{theorem}

\subsection{Enumeration of plane overpartitions}

We can get the
generating function for plane overpartitions whose largest entry
is at most $n$ from Theorem \ref{symmetric}. By this bijection, if $M$ is a  symmetric
matrix of size $2n\times 2n$ with blocks $A,B,C$ and $D$, each
of size $n\times n$ corresponding to a plane partition $\Pi \in \mathcal{L}(n)$, we have that
$$|\Pi|=\sum_{i,j} i(a_{ij}+b_{ij}+c_{ij}+d_{ij}) \quad \textrm{and}\quad o(\Pi)=\sum_{i,j} b_{ij}+d_{ij}.
$$
As $M$ is symmetric, i.e. $a_{ij}=a_{ji}$, $d_{ij}=d_{ji}$ and $b_{ij}=c_{ji}$, we can express the above formulas as
$$
|\Pi|=\sum_{i,j}(i+j)b_{ij}+\sum_{i} i(a_{ii}+d_{ii}) +\sum_{i<j}
(i+j)(a_{ij}+d_{ij})
$$
and
$$
o(\Pi)=\sum_{i,j} b_{ij}+2\sum_{i<j}d_{ij}+\sum_i d_{ii}.
$$

Let ${\mathcal{O}}_n(q,a)=\sum_{\Pi \in \mathcal{L}(n)} a^{o(\Pi)}q^{|\Pi|}$. 
We are now ready to prove Theorem \ref{boundedparts} which 
states that:
$$
{\mathcal O}_n(q,a)=\prod_{j=1}^n \frac{\prod_{i=0}^n(1+aq^{i+j})}
{\prod_{i=1}^{j}(1-q^{i+j-1})(1-a^2q^{i+j})}.
$$
Indeed,
\begin{eqnarray*}
{\mathcal O}_n(q,a)&=& \sum_{M \in \mathcal{M}_n} \prod_{j=1}^n\left(\prod_{i=1}^n (q^{i+j}a)^{b_{ij}}\right)\left(q^{ja_{jj}}(aq^j)^{d_{jj}}\right)
\left(\prod_{i=1}^{j-1}q^{(i+j)a_{ij}}(a^2q^{i+j})^{d_{ij}}\right)\\
&=& \prod_{j=1}^n\frac{\prod_{i=1}^n(1+aq^{i+j})}
{(1-q^j)(1-aq^j)\prod_{i=1}^{j-1}(1-q^{i+j})(1-a^2q^{i+j})}\\
&=& \prod_{j=1}^n\frac{\prod_{i=0}^{n}(1+aq^{i+j})}
{\prod_{i=0}^{j-1}(1-q^{i+j})\prod_{i=1}^j(1-a^2q^{i+j})}.\\
\end{eqnarray*}

When $n$ tends to infinity we get back the weighted generating function
of Proposition \ref{propweighted}. We can also get this result with
the method of the proof of Proposition \ref{propweighted} with the substitution
$x_i=q^{i-1}$ for $1\le i\le n+1$ and 0 otherwise and
$y_i=-aq^i$ for $1\le i\le n$ and 0 otherwise.


We can also get some more general results, as in \cite{BK}.
\begin{theorem}
The generating function for plane overpartitions whose parts lie in
a set $S$ of positive integers is given by:
\begin{equation*}
\prod_{i\in S}\left(\frac{\prod_{j\in
S}(1+aq^{i+j})}{(1-q^i)(1-aq^i)} \prod_{\begin{subarray}{c}j\in
S\\{j<i}\end{subarray}}\frac{1}{(1-q^{i+j})(1-a^2q^{i+j})}\right).
\end{equation*}
\end{theorem}

For example, as a corollary we get the generating function for plane overpartition with odd parts.  One can get a similar formula for even parts.
\begin{corollary}
The generating function for plane overpartitions with odd parts is
$$
\prod_{i=1}^\infty
\frac{(1+aq^{2i})^{i-1}}{(1-q^{2i-1})(1-aq^{2i-1})(1-q^{2i})^{\lfloor
i/2\rfloor}(1-a^2q^{2i})^{\lfloor i/2\rfloor}}.
$$
\end{corollary}

\section{Interlacing sequences and cylindric partitions} \label{s1}

We want to combine results of \cite{Bo} and \cite{V2} to obtain a
1--parameter generalization of the formula for the generating
function for cylindric partitions related to Hall-Littlewood
symmetric
functions. 

We use the definitions of interlacing sequences, profiles, cylindric
partitions, polynomials $A_\Pi(t)$ and $A_\Pi^{\text{cyl}}(t)$ given
in the introduction.

For an ordinary partition $\lambda$ we define a polynomial
\begin{equation}\label{formulab}
b_\lambda(t)=\prod_{i\geq1}\varphi_{m_i(\lambda)}(t),
\end{equation}
where $m_i(\lambda)$ denotes the number of times $i$ occurs as a part of $\lambda$ and $\varphi_r(t)=(1-t)(1-t^2)\cdots(1-t^r)$.


For a horizontal strip $\theta=\lambda/\mu$ we define
$$
\begin{array}{c}
I_{\lambda/\mu}=\left\{i\geq1|\,\theta_i^\prime=1 \text{ and } \theta_{i+1}^\prime=0 \right\}\\
J_{\lambda/\mu}=\left\{j\geq1|\,\theta_j^\prime=0 \text{ and } \theta_{j+1}^\prime=1 \right\}
\end{array}.
$$
Let
\begin{equation}\label{definphi}
\varphi_{\lambda/\mu}(t)=\prod_{i \in
I_{\lambda/\mu}}(1-t^{m_i(\lambda)}) \text{ and } \psi_{\lambda/\mu}(t)=\prod_{j \in
J_{\lambda/\mu}}(1-t^{m_j(\mu)}).
\end{equation}
Then
$$
\varphi_{\lambda/\mu}(t)/\psi_{\lambda/\mu}(t)=b_\lambda(t)/b_\mu(t).
$$

For an interlacing sequence $\Lambda=(\lambda^1,\dots,\lambda^T)$
with profile $A=(A_1,\dots,A_{T-1})$ we define $\Phi_{\Lambda}(t)$:
\begin{equation}\label{alternation}
\Phi_{\Lambda}(t)=\prod_{i=1}^{T-1} \phi_{i}(t),
\end{equation}
where
$$
\phi_{i}(t)=
\begin{cases}
\varphi_{\lambda^{i+1}/\lambda^i}(t),&A_i=0,\\
\psi_{\lambda^{i}/\lambda^{i+1}}(t),&A_i=1.
\end{cases}
$$

For $\Lambda=(\lambda^1,\dots,\lambda^T)$ and $M=(\mu^1,\dots,\mu^S)$ such that $\lambda^T=\mu^1$ we define
$$
\Lambda \cdot M=(\lambda^1,\dots,\lambda^T,\mu^2,\dots,\mu^S)
$$
and
$$
\Lambda \cap M=\lambda^T.
$$
Then
\begin{equation}\label{APhiproiz}
A_{\Lambda \cdot M}=\frac{A_{\Lambda}A_M}{b_{\Lambda \cap
M}},\ \ \ \ \Phi_{\Lambda \cdot
M}=\Phi_{\Lambda }\Phi_{M}.
\end{equation}

For an interlacing sequence $\Lambda=(\lambda^1,\dots,\lambda^T)$
with profile $A=(A_1,\dots, A_{T-1})$ we define the reverse
$\overline{\Lambda}=(\lambda^T,\dots,\lambda^1)$ with profile 
$\overline{A}=(1-A_{T-1},\dots,1-A_1)$. Then
\begin{equation}\label{APhirev}
A_{\overline{\Lambda}}=A_{\Lambda},\ \ \ \ 
\Phi_{\overline{\Lambda}}=\frac{b_{\lambda^1}\Phi_{\Lambda
}}{b_{\lambda^T}}.
\end{equation}

For an ordinary partition $\lambda$ we construct an interlacing
sequence $\left\langle \lambda \right\rangle=(\emptyset,
\lambda^1,\dots, \lambda^L)$ of length $L+1=\ell(\lambda)+1$, where
$\lambda^i$ is obtained from $\lambda$ by truncating the last $L-i$
parts. Then
\begin{equation}\label{APhiseq}
A_{\left\langle \lambda \right\rangle}=\Phi_{\left\langle \lambda \right\rangle}=b_\lambda.
\end{equation}

In \cite{V2} (Propositions 2.4 and 2.6) it was shown
shown that for a plane partition $\Pi$
\begin{equation}\label{APhiplane}
\Phi_{\Pi}=A_{\Pi}.
\end{equation}
The following proposition is an analogue of (\ref{APhiplane}) for interlacing sequences.

\begin{proposition} If  $\Lambda=(\lambda^1,\dots,\lambda^T)$ is an interlacing sequence then
$$
b_{\lambda^1}\Phi_{\Lambda}={A_{\Lambda}}.
$$
\end{proposition}
\begin{proof} If we show that the statement is true for sequences with constant profiles
then inductively using (\ref{APhiproiz}) we can
show that it is true for sequences with arbitrary profile. It is enough
to show that the statement is true for sequences with $(0, \dots,
0)$ profile because of (\ref{APhirev}). So, let
$\Lambda=(\lambda^1,\dots,\lambda^T)$ be a sequence with
$(0,\dots,0)$ profile. Then $\Pi=\left\langle
\lambda^1\right\rangle\cdot\Lambda\cdot\overline{\left\langle
\lambda^T\right\rangle}$ is a plane partition and from
(\ref{APhiproiz}), (\ref{APhirev}) and (\ref{APhiseq})  we obtain that $A_\Pi=A_\Lambda$ and
$\Phi_\Pi=b_{\lambda^1}\Phi_\Lambda$.  Then from (\ref{APhiplane}) it follows that $A_\Lambda=b_{\lambda^1}\Phi_\Lambda$.
\end{proof}
For skew plane partitions and cylindric partitions we obtain the
following two corollaries.
\begin{corollary}\label{corskew}
For a skew plane partition $\Pi$ we have $\Phi_\Pi=A_\Pi.$
\end{corollary}
\begin{corollary}\label{corcyl}
If $\Pi$ is a cylindric partition given by
$\Lambda=(\lambda^0,\dots,\lambda^T)$ then
$\Phi_{\Lambda}=A^{\text{cyl}}_{\Pi}$.
\end{corollary}
The last corollary comes from the observation that if a cylindric
partition $\Pi$ is given by a sequence
$\Lambda=(\lambda^0,\dots,\lambda^T)$ then
\begin{equation}
A^{\text{cyl}}_\Pi(t)=A_{\Lambda}(t)/b_{\lambda^0}(t).
\end{equation}

In the rest of this section we prove generalized MacMahon's formulas
for skew plane partitions and cylindric partitions that are stated
in Theorems \ref{skew} and \ref{uvodcyl}. The proofs of these theorems were inspired by \cite{Bo}, \cite{OR}
and \cite{V2}. We use a special class of symmetric functions called
Hall-Littlewood functions.

\subsection{\bf{The weight functions}} \label{s2.1.1}
In this subsection we introduce weights on sequences of ordinary
partitions. For that we use Hall-Littlewood symmetric functions $P$
and $Q$. We recall some of the facts about these functions, but for
more details see Chapters III and VI of \cite{Mac}. We follow the
notation used there.

Recall that Hall-Littlewood symmetric functions
$P_{\lambda/\mu}(x;t)$ and $Q_{\lambda/\mu}(x;t)$ depend on a
parameter $t$ and are indexed by pairs of ordinary partitions
$\lambda$ and $\mu$. In the case when $t=0$ they are equal to
ordinary Schur functions and in the case when $t=-1$ to Schur $P$
and $Q$ functions.

The relationship between $P$ and $Q$ functions is given by (see
(5.4) of \cite[Chapter III]{Mac})
\begin{equation}\label{RelQP}
Q_{\lambda/\mu}(x;t)=\frac{b_\lambda}{b_\mu}P_{\lambda/\mu}(x;t),
\end{equation}
where $b$ is given by (\ref{formulab}).
Recall that (by (5.3) of \cite[Chapter III]{Mac} and
(\ref{RelQP}))
\begin{equation}\label{nonzero}
P_{\lambda/\mu}=Q_{\lambda/\mu}=0\;\;\;\text{unless } \lambda\supset
\mu.
\end{equation}

We set $P_\lambda=P_{\lambda/\emptyset}$ and
$Q_\lambda=Q_{\lambda/\emptyset}$. Recall that ((4.4) of
\cite[Chapter III]{Mac})
$$
H(x,y;t):=\sum_{\lambda
}Q_\lambda(x;t)P_\lambda(y;t)=\prod_{i,j}\frac{1-tx_iy_j}{1-x_iy_j}.
$$

A specialization of an algebra $\mathcal{A}$ is an algebra
homomorphism $\rho: \mathcal{A} \to \mathbb{C}$. If $\rho$ and
$\sigma$ are specializations of the algebra of symmetric functions
then we write $P_{\lambda/\mu}(\rho;t)$, $Q_{\lambda/\mu}(\rho;t)$
and $H(\rho,\sigma;t)$ for the images of $P_{\lambda/\mu}(x;t)$,
$Q_{\lambda/\mu}(x;t)$ and $H(x,y;t)$ under $\rho$, respectively
$\rho \otimes \sigma$. Every map $\rho:(x_1,x_2,\dots)\to
(a_1,a_2,\dots)$ where $a_i\in \mathbb{C}$ and only finitely many
$a_i$'s are nonzero defines a specialization. These specializations
are called evaluations. A multiplication of a specialization $\rho$
by a scalar $a\in \mathbb{C}$ is defined by its images on power
sums:
$$
p_n(a \cdot \rho)=a^np_n(\rho).
$$

If $\rho$ is the specialization of $\Lambda$ where
$x_1=a,\,x_2=x_3=\ldots=0$ then by (5.14) and (5.14') of \cite[Chapter VI]{Mac}
\begin{equation} \label{spec}
\begin{array}{lcc}
 Q_{\lambda/\mu}(\rho;t)=
\begin{cases}
\varphi_{\lambda/\mu}(t) a^{|\lambda|-|\mu|} & \text{$\;\;\;\;\;\;\;\lambda \supset \mu$,  $\lambda/\mu$ is a horizontal strip},\\
0 & \text{$\;\;\;\;\;\;\;$otherwise}.
\end{cases}
\end{array}
\end{equation}
Similarly,
\begin{equation} \label{specP}
\begin{array}{lcc}
 P_{\lambda/\mu}(\rho;t)=
\begin{cases}
\psi_{\lambda/\mu}(t) a^{|\lambda|-|\mu|} & \text{$\;\;\;\;\;\;\;\lambda \supset \mu$,  $\lambda/\mu$ is a horizontal strip},\\
0 & \text{$\;\;\;\;\;\;\;$otherwise}.
\end{cases}
\end{array}
\end{equation}

Let $T \geq 2$ be an integer and $\rho^\pm=(\rho_1^{\pm},\ldots,
\rho_{T-1}^{\pm})$ be finite sequences of specializations. For a
sequence of partitions $ \Lambda=(\lambda^{1}, \ldots ,\lambda^{T})$
we set the weight function $W(\Lambda)$ to be
\begin{eqnarray*}
W(\Lambda)
&=&q^{T|\lambda^{T}|}\sum_{M}\prod_{n=1}^{T-1}P_{\lambda^{n}/\mu^{n}}
(\rho_{n}^{-};t)Q_{\lambda^{n+1}/\mu^{n}}({\rho_{n}^{+}};t),
\end{eqnarray*}
where $q$ and $t$ are parameters and the sum ranges over all
sequences of partitions $M=(\mu^{1}, \ldots ,\mu^{T-1})$.

We can also define the weights using another set of specializations
$R^{\pm}=(R_1^\pm,\dots,R_{T-1}^\pm)$ where $R_i^\pm=q^{\pm
i}\rho_i^\pm$. Then
$$
W(\Lambda)=\sum_M W(\Lambda,M, R^-,R^+),
$$
where
$$
W(\Lambda,M,
R^-,R^+)=q^{|\Lambda|}\prod_{n=1}^{T-1}P_{\lambda^{n}/\mu^{n}}
(R_{n}^{-};t)Q_{\lambda^{n+1}/\mu^{n}}(R^{+}_n;t).
$$
We will focus on two special sums, namely
$$
Z_{\text{skew}}=\sum_{\Lambda=(\emptyset,\lambda^1,\ldots,\lambda^T,\emptyset)}W(\Lambda)
$$
and
$$
Z_{\text{cyl}}=\sum_{\substack{\Lambda=(\lambda^0,\lambda^1,\ldots,\lambda^T)\\
\lambda^0=\lambda^T}}W(\Lambda).
$$

Let $A^-=(A^-_1,\dots,A^-_{T-1})$ and $A^+=(A^+_1,\dots,A^+_{T-1})$
be sequences of 0's and 1's such that $A^-_k+A^+_k=1$.

If the specializations $R^\pm$ are evaluations given by
\begin{equation}{\label{evaluations}}
{R}_k^\pm: x_1=A^\pm_k,x_2=x_3=\ldots=0, 
\end{equation}
then by (\ref{spec}) and (\ref{specP}) the weight $W(\Lambda)$ vanishes unless $\Lambda$ is an interlacing
sequence of profile $A^-$, and in that case
$$
W(\Lambda)=\Phi_{\Lambda}(t)q^{|\Lambda|}.
$$

Then, from Corollaries \ref{corskew} and \ref{corcyl}, we have 
$$
Z_{\text{skew}}=\sum_{\Pi \in \text{Skew}(T,A)}A_\Pi(t)q^{|\Pi|}
$$
and
$$
Z_{\text{cyl}}=\sum_{\Pi \in
\text{Cyl}(T,A)}A^{\text{cyl}}_\Pi(t)q^{|\Pi|},
$$
where $A^-$ in both formulas is given by the fixed profile $A$ of
skew plane partitions, respectively cylindric partitions.

\subsection{\bf{Partition functions}}
%
If $\rho$ is $x_1=r,\,x_2=x_3=\ldots=0$ and $\sigma$ is
$x_1=s,\,x_2=x_3=\ldots=0$ then
\begin{equation}\label{simpleh}
H(\rho,\sigma)=\frac{1-trs}{1-rs}.
\end{equation}
Thus, for the specializations $\rho_i^{\pm}=q^{\mp i}R_i^{\pm}$, where $R^{\pm}$ are given by (\ref{evaluations}), we have
\begin{equation}\label{h}
H(\rho_i^+,\rho_j^-)=\frac{1-tq^{j-i}A^+_iA^-_j}{1-q^{j-i}A^+_iA^-_j}.
\end{equation}

We now use Proposition 2.2 of \cite{V2}):
\begin{proposition}
$$
Z_{\text{skew}}(\rho^-,\rho^+)=\prod_{0\leq i<j\leq T}
H(\rho_i^+,\rho_j^-).
$$
\end{proposition}

This proposition together with (\ref{h}) implies Theorem
\ref{skew}. The generating function formula for skew plane partitions can
also be seen as the generating function formula for reverse plane partitions
as explained in the introduction.

Each skew plane partition can be represented as an infinite sequence
of ordinary partitions by adding infinitely many empty partitions to
the left and right side. In that way, the profiles become infinite
sequences of 0's and 1's. Theorem \ref{skew} also gives  the
generating function formula for skew plane partitions of infinite profiles
$A=(\dots,A_{-1},A_0,A_{1},\dots)$:
\begin{equation}\label{infiniteskew}
\sum_{\Pi \in \text{Skew}(A)}A_\Pi(t)q^{|\Pi|}= \prod_{\substack{i<
j\\A_i=0,\,A_j=1}}\frac{1-tq^{j-i}}{1-q^{j-i}}.
\end{equation}

Similarly for cylindric partitions, using (\ref{simpleh}) together with
the following proposition we obtain Theorem \ref{uvodcyl}.
\begin{proposition}\label{Z}
\begin{equation} \label{FormulaForZ}
Z_{\text{cyl}}(R^-,R^+)
=\prod_{n=1}^\infty
\frac{1}{1-q^{nT}}\prod_{k,\,l=1}^TH(q^{(k-l)_{(T)}+(n-1)T}R^-_k,R^+_l),
\end{equation}
where $i_{(T)}$ is the smallest positive integer such that $i\equiv
i_{(T)} \text{ mod }T$.
\end{proposition}

\begin{proof}
We use
\begin{equation}\label{zamena}
\sum_{\lambda }Q_{\lambda / \mu}(x)P_{\lambda /
\nu}(y)=H(x,y)\sum_{\tau }Q_{\nu / \tau}(x)P_{\mu / \tau}(y).
\end{equation}
The proof of this is analogous to the proof of Proposition 5.1 ((\ref{zamena}) for $t=-1$) that
appeared in \cite{V1}. Also, see Example 26 of
Chapter I, Section 5 of \cite{Mac}.

The proof of (\ref{FormulaForZ}) uses the same idea as in the proof
of Proposition 1.1 of \cite{Bo}.

We start with the definition of $Z_{\text{cyl}}(R^-,R^+)$: 
\begin{eqnarray*}
Z_{\text{cyl}}(R^-,R^+)&=&\sum_{\Lambda,M}W(\Lambda,M,R^-,R^+)\\
&=&
\sum_{\Lambda,M}q^{|\Lambda|}\prod_{n=1}^{T}P_{\lambda^{n-1}/\mu^{n}}
(R_n^-)Q_{\lambda^{n}/\mu^{n}}(R_n^+)\\
&=&\sum_{\Lambda,M}q^{|M|}\prod_{n=1}^{T}P_{\lambda^{n-1}/\mu^{n}}
(q R^-_n)Q_{\lambda^{n}/\mu^{n}}(R^+_n).
\end{eqnarray*}
If $x=(x_1,x_2,\dots,x_T)$ is a vector we define the shift as
$\operatorname{sh}(x)=(x_2,\dots,x_T,x_1)$. We set $R^{\pm}_0=R^{\pm}_T$, $\mu_0=\mu_T$ and $\nu_0=\nu_T$. 
If using the formula (\ref{zamena}) we
substitute the sums over the $\lambda^{i}$'s by the sums over the $\nu^{i-1}$'s
we obtain
\begin{eqnarray*}
Z_{\text{cyl}}(R^-,R^+)&=&H(q\operatorname{sh}R^-;R^+)\sum_{M,N}q^{|M|}Q_{\mu^1/\nu^0}(R^+_T),P_{\mu^0/\nu^0}(q
R^-_1) \cdot
\\
&&\cdot Q_{\mu^2/\nu^1}(R^+_1)P_{\mu^1/\nu^1}(q R^-_2)\cdots
Q_{\mu^0/\nu^{T-1}}(q R^+_{T-1})P_{\mu^{T-1}/\nu^{T-1}}(R^-_{T})\\
&=&H(q\operatorname{sh}R^-;R^+)\sum_{{M},{N}}
W(\operatorname{sh}M,N,q\operatorname{sh}R^-,R^+)\\
&=&H(q\operatorname{sh}R^-;R^+)Z_{\text{cyl}}(q \operatorname{sh}R^-,R^+).
\end{eqnarray*}
Since $\operatorname{sh}^T=\operatorname{id}$, if we apply the same
trick $T$ times, we obtain
\begin{eqnarray*}
Z_{\text{cyl}}(R^-,R^+)&=&\prod_{i=1}^T{H(q^i\operatorname{sh}^iR^-;R^+)}\cdot
Z_{\text{cyl}}(q^TR^-,R^+)\\
&=&\prod_{i=1}^T{H(q^i\operatorname{sh}^iR^-;R^+)}\cdot Z_{\text{cyl}}(sR^-,R^+),
\end{eqnarray*}
where $s=q^T$. Thus,
$$
Z_{\text{cyl}}(R^{-},R^{+})=\prod_{n=1}^\infty\prod_{i=1}^T
{H(q^{i+(n-1)T}\operatorname{sh}^iR^-;R^+)} \lim_{n\to
\infty}{Z_{\text{cyl}}(s^nR^-,R^+)}.
$$
From
$$
\lim_{n\to \infty}{Z_{\text{cyl}}(s^nR^-,R^+)}=\lim_{n\to
\infty}{Z_{\text{cyl}}(\text{trivial},R^+)}=\prod_{n=1}^\infty\frac{1}{1-s^n}
$$
and
\begin{eqnarray*}
\prod_{i=1}^T{H(q^i\operatorname{sh}^iR^-,R^+)}&=&
\prod_{l=1}^T\left[\prod_{k=l+1}^{T}H(q^{k-l}R^-_k,R^+_l)\prod_{k=1}^lH(q^{T+k-l}R^-_k,R^+_l)\right]\\
&=&\prod_{k,\,l=1}^TH(q^{(k-l)_{(T)}}R^-_k,R^+_l),
\end{eqnarray*}
we conclude that (\ref{FormulaForZ}) holds.
\end{proof}
Observe that if in Theorem \ref{uvodcyl} we let $T\to \infty$, i.e.
the circumference of the cylinder tends to infinity then we recover
(\ref{infiniteskew}).

\section{Concluding remarks}

In this paper, we determine the generating functions for plane overpartitions
with several types of constraints. In particular, we can compute the generating
function for plane overpartitions with at most $r$ rows and $c$ columns
and the generating function for plane overpartitions with entries at most $n$.
All the proofs can be done with symmetric functions, but we also highlight combinatorial
proofs when these are simple.
The natural question is therefore to put those constraints together and
to compute the generating
function of plane overpartitions with at most $r$ rows, $c$ columns
and entries at most $n$. Unfortunately, this generating function cannot be written
in terms of a product in the style of Theorems 1--8. For example, when $r=1$,
this generating function can be written as
$$
\sum_{k=0}^n a^k q^{k+1\choose 2}\frac{(q)_{n-k+c}}{(q)_{n-k}(q)_k(q)_{c-k}}.
$$
Computer experiments show that there is a large irreducible factor in the product.
Let $\mathbb A$ be the alphabet $1+q+\cdots +q^n$
and $\mathbb B$ be the alphabet $-aq-aq^2-\cdots -aq^n$.
We can write this generating function as
$$
\sum_{\lambda\subseteq c^r} s_{\lambda}(\mathbb A-\mathbb B) 
$$
where the sum is taken on partitions $\lambda$ such that all the parts of $\lambda'$
have the same parity as $r$.

In Section 6 we compute generating functions for cylindric partitions.
Their generating functions are elegant products. In the case $t=0$ (cylindric partitions
and $t=-1$ cylindric overpartitions, one could certainly adapt the ideas of Gansner \cite{G}
to give  a constructive proof of the result. It should be interesting to generalize the
combinatorial techniques used in this paper in Sections 3 and 5 to
understand the combinatorics of plane partitions for general $t$.

 \end{document}